\documentclass[a4paper,11pt]{article}

\usepackage[utf8]{inputenc}
\usepackage{amsfonts}
\usepackage{amssymb}
\usepackage{amsthm}
\usepackage{amsmath}
\usepackage{a4wide}
\usepackage{mathrsfs}
\usepackage{epsfig}
\usepackage{palatino}
\usepackage{tikz,fp,ifthen}
\usepackage{float}
\usepackage{comment}
\usepackage{enumitem} 

\usepackage[colorinlistoftodos,textwidth=2.3cm]{todonotes} 

\newcommand{\pdfgraphics}{\ifpdf\DeclareGraphicsExtensions{.pdf,.jpg}\else\fi}
\usepackage{graphicx}

\usepackage{color}
\definecolor{hanblue}{rgb}{0.27, 0.42, 0.81}
\definecolor{red}{rgb}{1.0, 0.0, 0.0}
\usepackage[colorlinks, citecolor=hanblue,linkcolor=blue, urlcolor = hanblue]{hyperref}

\usepackage{mathtools}
\mathtoolsset{showonlyrefs}

\numberwithin{equation}{section}

\theoremstyle{plain}

\newtheorem{teo}{Theorem}[section]
\newtheorem{lemma}[teo]{Lemma}
\newtheorem{prop}[teo]{Proposition}

\theoremstyle{definition}
\newtheorem{dfnz}[teo]{Definition}

\theoremstyle{remark}
\newtheorem{rem}[teo]{Remark}

\numberwithin{equation}{section}

\def\k{\boldsymbol{\kappa}}

\def\R{\mathbb R}

\newcommand{\e }{\varepsilon }

\newcommand{\ds }{\displaystyle}

\newcommand{\intbar}{\etaathop{\int\etaakebox(-13.5,0){\rule[4pt]{.7em}{0.3pt}}
\kern-6pt}\nolimits}

\newcommand{\be}{\begin{equation}}
\newcommand{\ee}{\end{equation}}
\newcommand{\bea}{\begin{equation*}}
\newcommand{\eea}{\end{equation*}}

\newcommand{\pas}{\partial_s}
\newcommand{\pat}{\partial_t}
\newcommand{\pax}{\partial_x}
\newcommand{\paz}{\partial_z}

\def\R{{{\mathbb R}}}

\def\pol{{\mathfrak{p}}}

\def\E{\mathcal{E}}

\newcommand{\N}{\mathbb{N}}

\def\be{\begin{equation}}
\def\ee{\end{equation}}
\def\bea{\begin{eqnarray*}}
\def\bean{\begin{eqnarray}}
\def\eean{\end{eqnarray}}
\def\eea{\end{eqnarray*}}

\DeclarePairedDelimiter\scal{\langle}{\rangle} 

\newcommand{\de}{{\,\mathrm{d} }} 

\newcommand{\pa}{\partial}

\begin{document}

\pdfgraphics 

\title{Elastic flow of curves with partial free boundary}

\author{Antonia Diana\footnote{Scuola Superiore Meridionale, Largo San Marcellino 10, 80138, Naples, Italy}}
\date{}
\maketitle

\begin{abstract}
\noindent
We consider a curve with boundary points free to move on a line in $\R^2$, which evolves by the $L^2$-gradient flow of the elastic energy, that is, a linear combination of the Willmore and the length functional.
For this planar evolution problem, we study the short and long-time existence.  
 Once we establish under which boundary conditions the PDE's system is well-posed (in our case the Navier boundary conditions), employing the Solonnikov theory for linear parabolic systems in H\"older space, we show that there exists a unique flow in a maximal time interval $[0,T)$. Then, using energy methods we prove that the maximal time is $T= + \infty$.
\end{abstract}

\bigskip

\tableofcontents

\section{Introduction}

In this paper, we consider the geometric evolution of a curve
with a partially free boundary.
To be more precise, we consider the gradient flow of the elastic energy under the constraint that 
the boundary points of the curve 
have to remain attached to the $x$-axis.

This paper fits within the broad range of topics on the geometric evolution of curves and surfaces, where the evolution law is dictated by functions of curvature. These topics have recently gained increasing attention from the mathematical community due to their applications to various physical problems and the fascinating challenges they present in analysis and geometry. 

\medskip
The elastic energy of a curve is a linear combination of the $L^2$-norm of its curvature $\boldsymbol{\kappa}$ (also known as one-dimensional Willmore functional) and its weighted length, namely
\begin{equation} 
\mathcal{E}(\gamma)
=\int_{\gamma}\vert\boldsymbol{\kappa}\vert^2
+\mu\de  s 
\end{equation}
where $\mu>0$. 

Before passing to the evolutionary problem, 
we say a few words about the critical points of the energy $\E$, known as {\em elasticae} or {\em elastic curves}. As explained in~\cite{Truesdell83}, elasticae have been studied since the time of Bernoulli and Euler, who used elastic energy as a model for the bending energy of elastic rods. Still later, Born, in his Thesis of $1906$, plotted the first figures of elasticae, using numerical schemes. However, in the last decades, many authors contribute to their classification, for instance, we refer to Langer and Singer~\cite{langersinger1, langersinger2}, Linn\'er~\cite{linner}, Djondjorov et al.~\cite{DjHaMlVa} and Langer and Singer~\cite{langersinger3}, Bevilacqua, Lussardi and Marzocchi~\cite{BevLusMar}, the same authors with Ballarin~\cite{BalBevLusMa}, for the case of a functional which depends both on the curvature and the torsion of the curve.
More recently, Miura and Yoshizawa in a series of papers~\cite{MiYo22, MiYo22-2, MiYo23}, give a complete classification of both clamped and pinned $p$-elasticae.

\medskip
In this paper, we aim to study the $L^2$-gradient flow of $\E$.
To the best of our knowledge, the problem was introduced by Polden in his PhD Thesis~\cite{polden2}, where it is shown that given as initial datum a smooth immersion of the circle in the plane, then there exists a smooth solution to the gradient flow problem for all positive times which sub-converges to an elastica. Then, Dziuk, Kuwert and Sch\"atzle generalized the global existence and sub-convergence result to $\R^n$ and derived an algorithm to treat the flow and compute several numerical examples. 
Later, the evolution of elastic curves has been extended and studied in detail both for closed curves (see for instance~\cite{kuschatdz,ManPoz,polden2,Poz22}) as well as for open curves with Navier boundary conditions in~\cite{NoOk14,NoOk17} and clamped boundary conditions in~\cite{DaPoSp16, NoOk17,Ch12,Sp17}.
We also recall that a slightly different problem was tackled, among others, by Wen in~\cite{wen2} and by Rupp and Spener in~\cite{Rupp2020ExistenceAC}, where the authors analyzed the elastic flow of curves with a nonzero rotation index and clamped boundary conditions respectively, which are in both cases subject to fix length and in~\cite{Koiso96,Ok07,Ok08} where a variety of constraints are considered. For the sake of completeness, we also mention that the $L^2$-gradient flow of $\int_{\gamma}\vert\boldsymbol{\kappa}\vert^2 \de  s $ for curve subjected to fix length is studied in~\cite{dalipo1, dalipo2,kuschatdz}, indeed other fourth (or higher) order flows are analyzed, for instance, in~\cite{DaPo14,Wh15,AbBu19, AbBu20,McWhWu17,McWhWu19,WhWh}.
Finally, we mention the survey~\cite{ManPluPozSurvey} for a complete review of the literature and we recommend all the references therein.

\medskip
As already said, 
in this paper, we let evolve a curve supposing that 
it remains attached to the $x$-axis.
To derive the flow, we start by writing
the associated Euler-Lagrange equations and in particular we find suitable ``natural'' boundary conditions for this problem (these boundary conditions
are known in the literature as Navier conditions).
We thus get that the evolution
can be described by solutions of a system
of quasilinear fourth order with boundary conditions in~\eqref{navierbc}, namely the attachment condition, second and third order conditions.
We then introduce a class of 
admissible initial curves
of class $C^{4+\alpha}$ with $\alpha \in  (0,1)$
which needs to be non-degenerate, in the sense that the $y$-component of the unit tangent vector must be positive at boundary points, and satisfy (in addition to the conditions mentioned above) an extra fourth order condition (see Definition~\ref{admissinitialgeom}).

Then, we establish well-posedness of the flow. More precisely, starting with a (geometrically) admissible initial curve  we prove in Theorem~\ref{geomexistence} that 
there exists a unique (up to reparametrization) solution to the flow 
in a small time interval $[0,T]$ with $T>0$, 
that can be described by a parametrization of class $C^{{\frac{4+\alpha}{4}},4} ([0, T ] \times [0, 1])$.

To do so, we choose a specific tangential velocity turning the system~\eqref{evolutionlaw} into a non-degenerate parabolic boundary value problem without changing the geometric nature of the evolution (namely the {\em analytic problem}~\eqref{analyticprobl}). Then, we solve the {\em analytic problem} using a linearization procedure and a fixed point argument. The main difficulty is actually to solve the associated {\em linear system}~\eqref{linearprobl}, coupled with extra compatibility conditions (see Definition~\ref{lincompcond}), employing the Solonnikov theory for linear parabolic systems in H\"older space introduced in~\cite{solonnikov1}, as it is shown in Theorem~\ref{linearexistence}.

Once we have a solution for the {\em analytic problem}, the key point is to ensure that solving~\eqref{analyticprobl} is enough to obtain a unique solution to the original {\em geometric problem}. This is shown in Theorem~\ref{geomexistence}, following the approach presented in~\cite{garkcoh} and later in~\cite{GaMePl1}.

The second natural step is trying to understand the long-time behavior of the evolving curves. This leads to our main result.

\begin{teo}\label{mainthm}
 Let $\gamma_0$ be a geometrically admissible initial curve and $\gamma_t$ be a solution to the elastic flow with initial datum $\gamma_0$ in the maximal time interval $[0,T)$ with $T \in (0, \infty) \cup \{\infty\}$. Then, up to reparametrization and translation of $\gamma_t$, it follows
$$T= \infty$$
or at least one of the following holds
\begin{itemize}
\item the inferior limit of the length of $\gamma_t$ is zero as $ t \to T$;
\item the inferior limit of the $y$-component of the unit tangent vector at the boundary is zero as $ t \to T$.
\end{itemize}
\end{teo}
Even though the structure of the proof of this result is based on a contradiction argument
already present in the literature 
(see for instance~\cite{polden2, kuschatdz, ManPluPozSurvey, DaChPo19,GaMePl2}) this is the most technical 
part of the paper and it contains relevant novelties.
\\We find energy type inequalities, more precisely bounds on the $L^2$-norm of the second and sixth derivative of the curvature, which leads to contradicting the finiteness of $T$.
Those estimates, which involved the smallest number of derivatives, can be derived under the assumption that during the evolution the length is uniformly bounded away from zero and that the curve remains non-degenerate in a uniform sense (see Definition~\ref{defunifnondeg}). 
\\Moreover, we underline that only estimates for geometric quantities, namely the curvature, are needed. In particular, the proof itself is independent of the choice of tangential velocity which corresponds to the very definition of the flow, where only the normal velocity is prescribed. For this reason, following~\cite{DaChPo19}, we reparametrize the flow in such a way that the tangential velocity linearly interpolates its values at boundary points (see condition~\eqref{Lambda}) and such that suitable estimates both inside and at boundary points hold. 
With this choice and the uniform bounds for the curvature, we can extend the flow smoothly up to the time $T$ given by the short-time existence result and then restart the flow, contradicting the maximality of $T$.
\\In short, our approach combines the one presented in~\cite{DaChPo19} and the other in~\cite{GaMePl2}, in the sense that we choose a tangential velocity as explained above and we use the minimum number of derivatives (and hence of estimates) which are needed to conclude the proof of Theorem~\ref{mainthm}.

\medskip
This work is organized as follows: in the next section we formulate the geometric evolution problem for elastic curves and we show that those curves decrease the energy $\E$. In Section~$3$ we show short-time existence of a unique smooth solution using the Solonnikov theory and a contraction argument. 
We also show geometric uniqueness. In the final Section~$5$, we prove the long-time existence result using the curvature bounds provided in Section~$4$.

\section{The elastic flow}

\subsection{Preliminary definitions and notation}

A regular curve is a continuous map $\gamma : [a,b] \to \R^2$ which is differentiable on $(a,b)$ and such that $\vert \pax \gamma \vert$ never vanishes on $(a,b)$. Without loss of generality, from now on we consider $[a,b]=[0,1]$. 

We denote by $s$ the arclength parameter, 
then $\pas:=\frac{1}{\vert \pax \gamma \vert}\pax $ and $\de s:= \vert \pax \gamma \vert \de x$ are the derivative and the measure with respect to the arclength parameter of the curve $\gamma$, respectively. 
\medskip \\{\em From now on, we will pass to the arclength parametrization of the curves without further comments.} 
\medskip\\If we assume that $\gamma$ is a regular 
planar curve of class at least $C^1$, 
we can define the unit tangent vector $\tau=\vert \pax \gamma \vert^{-1} \pax \gamma$ and the unit normal vector $\nu$ as the anticlockwise rotation by $\pi/2$ of the unit tangent vector. 
\medskip\\We introduce the operator $\partial_s^\perp$ 
that acts on vector fields $\varphi$
defined as the normal component of $\partial_s\varphi$ along the curve $\gamma$, that is
$\partial_s^\perp\varphi=\partial_s\varphi
-\left\langle \partial_s\varphi,\partial_s\gamma\right\rangle\partial_s\gamma$. Moreover, for any vector $\psi(\cdot) \in \R^2$, we use the notation ($\psi(\cdot)_1$,$\psi(\cdot)_2)$ to denote the projection on the $x$-axis and $y$-axis, respectively.
\medskip\\ Let $\mu>0$. Assuming that $\gamma$ is of class $H^2$, we denote by $\boldsymbol{\kappa} = \pas \tau$ the curvature vector and we define the {\em elastic energy with a length penalization}
\begin{equation} 
\mathcal{E}(\gamma)
=\int_{\gamma}\vert\boldsymbol{\kappa}\vert^2
+\mu\de  s \, .
\end{equation}
Denoting by $k$ the oriented curvature, by means of relation $\boldsymbol{\kappa}= k \nu$ which holds in $\R^2$, the energy functional can be equivalently written as
\be\label{energy}
\E(\gamma)=\int_{\gamma}k^2+\mu\de  s \, .
\ee

\subsection{Formal derivation of the flow}
Let $\gamma:[0,1]\to\mathbb{R}^2$ be a regular curve
of class $H^2$.
 We consider a variation $\gamma_\varepsilon=\gamma+\varepsilon\psi$ with $\varepsilon\in\mathbb{R}$ and $\psi:[0,1]\to\mathbb{R}^2$ of class $H^2$, which is regular whenever $\vert \varepsilon\vert$ is small enough.
By direct computations (see~\cite{ManPoz}, for instance) we get the {\em first variation} of $\E$, that is
\begin{align}
 \frac{d}{d\varepsilon}\mathcal{E}(\gamma_\varepsilon)\Big\vert_{\varepsilon=0}
 & =\int_{\gamma} 2 \langle \boldsymbol{\kappa},
 \partial^2_{s} \psi \rangle \de  s   
 + \int_{\gamma} (-3 |\boldsymbol{\kappa}|^2+\mu)
\left\langle \tau,\partial_s\psi\right\rangle \de  s\,.\label{primostep}
\end{align}
We say that a regular curve $\gamma$ of class $H^2$ is a {\em critical} point of $\E$ if for any $\psi$ its first variation vanishes. 

\begin{lemma}[Euler-Lagrange equations]
Let $\gamma:[0,1]\to\mathbb{R}^2$ be a critical point of $\E$ parametrized proportional to arclength.
Then, $\gamma$ is smooth and satisfies 
\be 2(\partial^\perp_s)^2 \boldsymbol{\kappa}+ |\boldsymbol{\kappa}|^2 \boldsymbol{\kappa} 
 -\mu \boldsymbol{\kappa}\, =0 
 \ee
in $(0,1)$. Moreover, if the 
endpoints are
constrained to the $x$-axis, the following 
Navier boundary conditions are fulfilled
\begin{align}
\begin{cases}\label{navierbcincomplete}
k(y)=0 &\text{curvature or second order conditions}
\\
\left(-2\pas k(y) \nu(y)
+\mu\tau(y)\right)_1 =0 &\text{third order conditions}
\end{cases}
\end{align}
for $y\in\{0,1\}$.
\end{lemma}
\begin{proof}
By a standard bootstrap argument, one can show that critical points of $\E$ are actually smooth 
(for the reader's convenience a proof of this fact
is given in Proposition~\ref{propregularitycriticalpoint} in the appendix).
Hence, integrating 
by parts the expression~\eqref{primostep}, we have
\begin{align}
 \frac{d}{d\varepsilon}\mathcal{E}(\gamma_\varepsilon) \Big\vert_{\e=0}
 =&
 \int_{\gamma}  \left\langle
 2 (\partial^\perp_s)^2 \boldsymbol{\kappa}+ |\boldsymbol{\kappa}|^2 \boldsymbol{\kappa} 
 -\mu \boldsymbol{\kappa},
\psi\right\rangle \de  s\nonumber\\
 +& 2 \left. \langle \boldsymbol{\kappa},
 \partial_s \psi\rangle \right|_0^1 
+  \left. \langle -2\partial^\perp_s \boldsymbol{\kappa}
- |\boldsymbol{\kappa}|^2\tau+\mu\tau, \psi\rangle \right|_0^1 \, . \label{firstvarBC} 
 \end{align} 
 Since $\gamma$ is critical, from formula~\eqref{firstvarBC} we immediately get  
 \be 2(\partial^\perp_s)^2 \boldsymbol{\kappa}+ |\boldsymbol{\kappa}|^2 \boldsymbol{\kappa} 
 -\mu \boldsymbol{\kappa}\, =0 
 \ee
and
 \be  2 \left. \langle \boldsymbol{\kappa},
 \partial_s \psi\rangle \right|_0^1 
+  \left. \langle -2\partial^\perp_s \boldsymbol{\kappa}
- |\boldsymbol{\kappa}|^2\tau+\mu\tau, \psi\rangle \right|_0^1=0 \, .\label{criticalpointbc}
\ee
We now recall that
\be\label{pasperpk}
\pas^\perp \boldsymbol{\kappa}=\pas \boldsymbol{\kappa}+ \vert \boldsymbol{\kappa} \vert^2 \tau \, .
\ee
Hence, from $\boldsymbol{\kappa}= k \nu$ and the Serret-Frenet equation in the plane, that is 
\be\label{SerretFrenet}
\pas \nu = - k \tau\,,
\ee 
the boundary terms in~\eqref{criticalpointbc} reduce to
 \be 2 \left. \langle k\nu,
 \partial_s \psi\rangle \right|_0^1 + \left \langle -2\partial_s k \nu- k^2 \tau+ \mu \tau, \psi \rangle \right |_0^1\,.\label{boundarypart}
\ee
The fact that 
 the endpoints must remain attached to the $x$-axis
 affects the class of test functions:
we can only consider variations $\gamma_\varepsilon=\gamma+\varepsilon\psi$
with 
\be\label{psiconstraint}
\psi(0)_2=\psi(1)_2=0\,.
\ee
Now,
letting first $\psi(0)_1 =\psi(1)_1= 0$, it remains the boundary term
$$2 \left. \langle k\nu,
 \partial_s \psi\rangle \right|_0^1=0\, ,$$
where the test functions $\psi$ appear differentiated. So, we can choose a test function $\psi$ such that 
\be
\pa_s \psi(0) = \nu (0) \qquad \text{and} \qquad \pa_s \psi(1)=0
\ee 
and we get $k(0)=0$. Then, interchanging the role of $\pa_s \psi(0)$ and $\pa_s \psi(1)$, we have $k(1)=0$.
\\It remains to consider the last term 
$$
\left \langle -2\partial_s k \nu- k^2 \tau+ \mu \tau, \psi \rangle \right |_0^1=0\, .$$
Taking into account the condition $k(0)=k(1)=0$, by arbitrariness of $\psi$
the term  is zero if
\be \left(-2\pas k(y) \nu(y)
+\mu\tau(y)\right)_1 =0
\ee
for $y \in \{0,1\}$.
\end{proof}

The previous lemma allows us to formally define the elastic flow of a curve with endpoints
constrained to the $x$-axis coupling the motion equation
  \be \label{motioneqkvett}
  \partial_t\gamma =-2(\partial^\perp_s)^2 \boldsymbol{\kappa}- |\boldsymbol{\kappa}|^2 \boldsymbol{\kappa} 
 +\mu \boldsymbol{\kappa} \,,
 \ee
 with the following 
Navier boundary conditions
 \begin{align}\label{navierbc}
 \begin{cases}
\gamma(y)_2=0 \quad&\text{attachment conditions}
\\k(y)=0 &\text{curvature or second order conditions}
\\\left(-2\pas k(y) \nu(y)
+\mu\tau(y)\right)_1 =0 &\text{third order conditions}
\end{cases}
\end{align}
for $y \in \{0,1\}$.

\subsection{Definition of the geometric problem}
In this section, we briefly introduce the parabolic H\"{o}lder spaces (see~\cite{solonnikov1} for more details).
\\Given a function $u:[0,T]\times [0,1]\to\mathbb{R}$, for $\rho\in (0,1)$ we define the semi-norms 
$$
[ u]_{\rho,0}:=\sup_{(t,x), (\tau,x)}\frac{\vert u(t,x)-u(\tau,x)\vert}{\vert t-\tau\vert^\rho}\,,
$$
and
$$
[ u]_{0,\rho}:=\sup_{(t,x), (t,y)}\frac{\vert u(t,x)-u(t,y)\vert}{\vert x-y\vert^\rho}\,.
$$
Then, for $l\in \{0,1,2,3,4\}$ and $\alpha\in (0,1)$, the parabolic H\"older space
$$
C^{\frac{l+\alpha}{4}, l+\alpha}([0,T]\times[0,1])
$$
is the space 
of all functions $u:[0,T]\times [0,1]\to\mathbb{R}$ that have continuous derivatives $\partial_t ^i\partial_x^ju$ where $i,j\in\mathbb{N}$ are such that $4i+j\leq l$ for which the norm
\begin{equation*}
\left\lVert u\right\rVert_{{\frac{l+\alpha}{4},l+\alpha}}:=\sum_{4i+j=0}^l\left\lVert\partial_t ^i\partial_x^ju\right\rVert_\infty
+\sum_{4i+j=l}\left[\partial_t ^i\partial_x^ju\right]_{0,\alpha}+\sum_{0<l+\alpha-4i-j<4}\left[\partial_t ^i\partial_x^ju\right]_{\frac{l+\alpha-4i-j}{4},0}
\end{equation*}
is finite. Moreover, the space $C^{\frac{\alpha}{4},\alpha}\left([0,T]\times[0,1]\right)$ coincides with the space 
$$
C^{\frac{\alpha}{4}}\left([0,T];C^0([0,1])\right)\cap C^0\left([0,T];C^\alpha([0,1])\right)\,,
$$
with equivalent norms.




\begin{dfnz}[Admissible initial curve]\label{admissinitialgeom}
    A regular curve $\gamma_0:[0,1] \to \R^2$ is an {\em admissible initial curve} for the elastic flow if 
    \begin{enumerate}
        \item it admits a parametrization which belongs to $C^{4+\alpha}([0,1], \R^2)$ for some $\alpha \in (0,1)$;
        \item it satisfies the Navier boundary conditions in~\eqref{navierbc}: attachment, curvature and third order conditions;
        \item  it satisfies the {\em non-degeneracy condition}, that is, there exists $\rho > 0$ such that 
        \be\label{nondeg} 
        (\tau_0(y))_2 \geq \rho \qquad \text{for $y \in \{0,1\} \, .$}
        \ee
        \item it satisfies the following {\em fourth order condition}
        \be ((-2\partial_s^2k_0(y)-k_0^3(y)+k_0(y))\nu_0(y))_2=0 \qquad \text{for $y \in \{0,1\} \, .$}
        \ee
    \end{enumerate}
    \end{dfnz}
\begin{dfnz}[Solution  of the geometric problem]\label{Def:elasticflow}
Let $\gamma_0$ be an admissible initial curve as in Definition~\ref{admissinitialgeom}
and $T>0$. 
A time-dependent family of curves
$\gamma_t$ for ${t\in [0,T]}$ is a solution to the \emph{elastic flow} with initial datum $\gamma_0$ in the maximal time interval $[0,T]$,
if there exists a parametrization 
$$
\gamma(t,x)\in 
C^{\frac{4+\alpha}{4}, 4+\alpha}\left([0,T]\times [0,1],\mathbb{R}^2\right)\,,
$$
with $\gamma$ regular
and such that for every $t\in [0,T], x\in [0,1]$ the system 
\begin{equation}
\begin{cases}
(\partial_t\gamma)^\perp=\left(-2\partial_s^2k-k^3+ \mu k\right)\nu\\
\gamma(0,x)=\gamma_0(x)\,,
\end{cases}\label{evolutionlaw}
\end{equation}
 coupled with boundary conditions~\eqref{navierbc}, is satisfied.
\end{dfnz}
\begin{rem}
The motion equation in~\eqref{evolutionlaw} follows from~\eqref{motioneqkvett}, using Serret-Frenet equation~\eqref{SerretFrenet} and recalling that 
\be
(\pas^\perp)^2 \boldsymbol{\kappa}= \pas^2 \boldsymbol{\kappa}+3 \left\langle \pas \boldsymbol{\kappa}, \boldsymbol{\kappa} \right \rangle \tau + \vert \boldsymbol{\kappa} \vert^2 \boldsymbol{\kappa} \, .
\ee
\end{rem}
\begin{rem}
    Observe that the formulation of the problem given so far involves purely geometric
quantities and hence it is invariant under reparametrizations.
Thus, given a solution $\gamma$ of~\eqref{evolutionlaw}, any reparametrization of $\gamma$ still satisfies system~\eqref{evolutionlaw}.
\end{rem}

\begin{rem}
As the authors pointed out in~\cite{ManPluPozSurvey}, in system~\eqref{evolutionlaw} only the normal component of the velocity is prescribed. This does not mean that the tangential velocity is necessarily zero. Indeed, we can equivalently write the motion equations as
    \begin{equation}\label{motionequationtang}
\partial_t\gamma=V\nu+\Lambda\tau\,,
\end{equation}
where $V=-2\partial_ s^2 k-k^3+\mu k$ and $\Lambda$ is some at least continuous function. 
\end{rem}

\subsection{Energy monotonicity}

In Proposition~\ref{energydecreases} we show that the energy of an evolving curve
decreases in time, adapting the proof of~\cite[Proposition~2.20]{ManPluPozSurvey}.

\begin{lemma}\label{evoluzionigeom}
If $\gamma$ satisfies~\eqref{motionequationtang}, the commutation rule
\begin{equation*}
\partial_{t}\partial_{s}=\partial_{s}\partial_{t}+\left(kV-\partial_s \Lambda\right)\partial_{s}\,
\end{equation*}
holds and the measure $\de s$ evolves as
\begin{equation}\label{metricaevol}
\partial_t(\de s)=\left(\partial_s \Lambda-kV\right)\de s\,.
\end{equation}
Moreover the unit tangent vector, unit normal vector, and the 
$j$-th
derivatives of scalar curvature of $\gamma$ satisfy
\begin{align}
\partial_{t}\tau&=\left(\partial_s V+\Lambda k\right)\nu\,,\label{pattau}\\
\partial_{t}\nu&=-\left(\partial_s V+\Lambda k\right)\tau\,,\label{patnu}\\
\partial_tk&=\left\langle \partial_{t}\boldsymbol{\kappa},\nu\right\rangle =\partial_s^2 V+\Lambda\partial_s k+k^{2}V\,\nonumber\\
&=-2\partial_{s}^{4}k-5k^{2}\partial_{s}^{2}k-6k\left(\partial_{s}k\right)^{2}
+\Lambda\partial_{s}k-k^{5}+\mu \left( \partial_{s}^{2}k+k^3\right)\,,\label{kt}
\end{align}
\end{lemma}
\begin{proof}
The proof of the lemma is obtained by direct computations, we refer for instance to~\cite[Lemma~2.19]{ManPluPozSurvey}.
\end{proof}

\begin{prop}\label{energydecreases}
Let $\gamma_t$ be a 
solution to the elastic flow in the sense of Definition~\ref{Def:elasticflow}. 
Then
\begin{align*}
\partial_{t}\mathcal{E}(\gamma_t)&=-\int_{\gamma} V^2\de  s\,.
\end{align*}
\end{prop}

\begin{proof}
Using the evolution laws collected in Lemma~\ref{evoluzionigeom}, we get
\begin{align*}
\partial_{t}\int_{\gamma}k^{2}+\mu\de  s
&=\int_{\gamma}2k\partial_{t}k+\left(k^{2}+\mu\right)\left(\partial_s \Lambda-kV\right)\de  s\\
&=\int_{\gamma}2k\left(\partial_s^2V+ \partial_s k \Lambda+k^{2}V\right)+\left(k^{2}+\mu\right)\left(\partial_s \Lambda-kV\right)\de  s\\
&=\int_{\gamma}^{}2k\partial_s^2V+k^3V-\mu kV+\partial_s\left(\Lambda\left(k^2+\mu\right)\right)\de  s\,.
\end{align*}
Integrating twice by parts the term $\int_{\gamma} 2k\partial_s^2V \de  s $ we obtain 
\begin{equation}\label{eq:perparti}
\partial_{t}\int_{\gamma}k^{2}+\mu\de  s
=-\int_{\gamma}^{}V^2\de  s+
\left ( 2k\partial_s V-2\partial_s kV+\Lambda(k^2+\mu)
  \right) \Big|_{0}^1\,.
\end{equation}
It remains to show that the contribution of the boundary term in~\eqref{eq:perparti} is zero, once we assume that Navier boundary conditions hold.

Since $k(y)=0$ for $y\in\{0,1\}$, we only need to show that
\be
-2\partial_s kV+ \mu \Lambda
\Big|_{0}^1 = 0 \, .
\ee
From $\gamma(y)= (\gamma_1 (y),0)$, using relation~\eqref{navierbc} we obtain
\begin{align}
0 =&  \langle \pa_t \gamma(y) , -2 \pa_s k(y) \nu(y)+ \mu \tau (y) \rangle 
\\=& \langle V(y) \nu (y) + \Lambda(y) \tau(y) , -2 \pa_s k(y) \nu(y)+ \mu \tau (y) \rangle
\\ =& -2 \pa_s k(y) V(y) + \mu \Lambda(y) \, ,
\end{align}
where $y \in \{ 0,1 \}$.
\end{proof}
\section{Short-time existence}
In this section we show that, fixed an admissible initial curve, there exists a maximal existence time $T$. 
To do so, we find a unique solution to the associated analytic problem defined in~\eqref{analyticprobl} using a standard linearization procedure. More precisely, we use Solonnikov theory (see~\cite{solonnikov1}) to prove the well-posedness of the linearized system and then we conclude with a fixed point argument. 
Then, a key point is to ensure that solving the analytic problem is enough to obtain a solution to the geometric problem~\eqref{evolutionlaw} and that the solution of~\eqref{evolutionlaw} is unique up to reparametrization.

\subsection{Definition of the analytic problem}
Let $T>0$ and $\alpha\in(0,1)$. Let us consider a time-dependent family of curves parametrized by a map
$\gamma \in C^{\frac{4+\alpha}{4},4+\alpha}([0,T]\times[0,1])$.

We compute the normal velocity of such moving curves in terms of the parametrization (see~\cite{GaMePl1} for more details), that is
\begin{align}
(\pat \gamma)^\perp=&-2\frac{\pax^4 \gamma  }{\left|\pax\gamma  \right|^{4}}
+12\frac{\pax^3 \gamma\left\langle \pax^2 \gamma,\pax\gamma\right\rangle }{\left|\pax\gamma\right|^{6}}
+5\frac{\pax^2 \gamma\left|\pax^2 \gamma\right|^{2}}{\left|\pax\gamma\right|^{6}}
+8\frac{\pax^2 \gamma  \left\langle \pax^3 \gamma  ,\pax\gamma  \right\rangle }
{\left|\pax\gamma  \right|^{6}}
-35\frac{\pax^2 \gamma  \left\langle \pax^2 \gamma  ,\pax\gamma  \right\rangle ^{2}}
{\left|\pax\gamma  \right|^{8}}\nonumber\\
&+\left\langle 2\frac{\pax^4 \gamma  }{\left|\pax\gamma  \right|^{4}}-12\frac{\pax^3 \gamma  \left\langle 
\pax^2 \gamma  ,\pax\gamma  \right\rangle }{\left|\pax\gamma  \right|^{6}}-5\frac{\pax^2 \gamma  \left|
\pax^2 \gamma  \right|^{2}}{\left|\pax\gamma  \right|^{6}}-8\frac{\pax^2 \gamma  \left\langle 
\pax^3 \gamma  ,\pax\gamma  \right\rangle }{\left|\pax\gamma  \right|
^{6}}+35\frac{\pax^2 \gamma  \left\langle \pax^2 \gamma  ,\pax\gamma  \right\rangle ^{2}}{\left|\pax\gamma  \right|^{8}},\tau  \right\rangle 
\tau  \nonumber\\
&+\mu\frac{\pax^2 \gamma  }{\left|\pax\gamma  \right|^{2}}
-\left\langle \mu\frac{\pax^2 \gamma  }{\left|\pax\gamma  \right|^{2}},
\tau  \right\rangle \tau  \,. \label{Vparam}
\end{align}
We now aim to use a well-known technique, which was introduced for the first time by DeTurck in~\cite{deturck} for the Ricci flow and then has been employed in a large variety of situations (see for instance~\cite{DaChPo19,GaMePl2,ManPluPozSurvey}).
\medskip

More precisely, we choose as tangential velocity the function
\begin{align}
\widetilde{\Lambda}:= \left\langle \right. & \left.  -2\frac{\pax^4 \gamma  }{\left|\pax\gamma  \right|^{4}}+12\frac{\pax^3 \gamma  \left\langle 
\pax^2 \gamma  ,\pax\gamma  \right\rangle }{\left|\pax\gamma  \right|^{6}}+5\frac{\pax^2 \gamma  \left|
\pax^2 \gamma  \right|^{2}}{\left|\pax\gamma  \right|^{6}}+8\frac{\pax^2 \gamma  \left\langle 
\pax^3 \gamma  ,\pax\gamma  \right\rangle }{\left|\pax\gamma  \right| 
^{6}}  \right. \\ & \left. -35\frac{\pax^2 \gamma  \left\langle \pax^2 \gamma  ,\pax\gamma  \right\rangle ^{2}}{\left|\pax\gamma  \right|^{8}} +\mu\frac{\pax^2 \gamma  }{\left|\pax\gamma  \right|^{2}},\tau  \right\rangle \, ,
\label{Tparam}
\end{align}
turning~\eqref{motionequationtang} into a non–degenerate equation
\begin{align}
    \pat \gamma =& V\nu +\widetilde{\Lambda}\tau
    \\=& -2\frac{\pax^4 \gamma  }{\left|\pax\gamma  \right|^{4}}
+12\frac{\pax^3 \gamma\left\langle \pax^2 \gamma,\pax\gamma\right\rangle }{\left|\pax\gamma\right|^{6}}
+5\frac{\pax^2 \gamma\left|\pax^2 \gamma\right|^{2}}{\left|\pax\gamma\right|^{6}}
+8\frac{\pax^2 \gamma  \left\langle \pax^3 \gamma  ,\pax\gamma  \right\rangle }
{\left|\pax\gamma  \right|^{6}}
\nonumber\\
&-35\frac{\pax^2 \gamma  \left\langle \pax^2 \gamma  ,\pax\gamma  \right\rangle ^{2}}{\left|\pax\gamma  \right|^{8}}+\mu\frac{\pax^2 \gamma  }{\left|\pax\gamma  \right|^{2}}\, .
\label{motionequationparam}\end{align}
Moreover, we specify another tangential condition 
\be\label{cond2ordboundary}
\langle \pax^2 \gamma (y) , \tau(y) \rangle =0 \qquad \text{for $y \in \{ 0,1\}$}
\ee
and we notice that this together with the curvature condition, is equivalent to the {\em second order condition}
    \be\label{secordcond}
    \pax^2 \gamma (y) =0  \qquad \text{for $y \in \{ 0,1\}$.}
    \ee
{\em From now on, we identify the curve with its parametrization without further comments.}
\begin{dfnz}[Admissible initial parametrization]\label{admissinitialanalytic}
 A map $\gamma_0:[0,1] \to \R^2$ is an {\em admissible initial parametrization} if 
    \begin{enumerate}
        \item it belongs to $C^{4+\alpha}([0,1], \R^2)$ for some $\alpha \in (0,1)$;
        \item it satisfies the Navier boundary conditions in~\eqref{navierbc}: attachment,  curvature and third order conditions;
        \item it satisfies the {\em non-degeneracy condition}~\eqref{nondeg};
        \item it satisfies the following {\em fourth order condition}
    \be
       \left(V(0,y)\nu_0(y)+\widetilde{\Lambda} (0,y) \tau_0(y) \right)_2=0 \qquad \text{for $y \in \{0,1\}$,}
       \ee
         where $\nu_0$, $\tau_0$ are the normal and tangent unit vectors to $\gamma_0$.
    \end{enumerate}
\end{dfnz} 
In the following, we refer to conditions $(2)-(4)$ in Definition~\ref{admissinitialanalytic} as {\em compatibility conditions}. 
\begin{dfnz}[Solution  of the analytic problem]\label{Def:elasticflowanalytic}
Let $\gamma_0$ be an admissible initial parametrization as in Definition~\ref{admissinitialanalytic}. A time-dependent parametrization 
$\gamma_t$ for ${t\in [0,T]}$ is a solution to the \emph{analytic elastic flow} with initial datum $\gamma_0$ in the time interval $[0,T]$ with $T>0$,
if 
$$
\gamma(t,x)\in 
C^{\frac{4+\alpha}{4}, 4+\alpha}\left([0,T]\times [0,1],\mathbb{R}^2\right)\,,
$$
with $\gamma$ regular
and such that for every $t\in [0,T], x\in [0,1]$ and $y \in \{0,1\}$, satisfies the system
  \be
  \begin{cases}
  \pat \gamma =  V\nu+\widetilde{\Lambda}\tau=-2\frac{\pax^4 \gamma  }{\left|\pax\gamma  \right|^{4}} + l.o.t. 
      \\ \gamma(y)_2=0 \quad&\text{attachment conditions,}
\\\pax^2 \gamma (y)=0 &\text{second order conditions,}
\\\left(-2\pas k(y) \nu(y)
+\mu\tau(y)\right)_1 =0 &\text{third order conditions,}
\\ \gamma(0,\cdot)= \gamma_0(\cdot) &\text{initial condition.}
  \end{cases} \label{analyticprobl}
  \ee
  \end{dfnz}

\subsection{Linearization}
This section is devoted to proving the existence and 
uniqueness of solutions
to the linearized system associated to~\eqref{analyticprobl}.
To do so, we show that the linearized
system can be solved using the general theory introduced by Solonnikov in~\cite{solonnikov1}.
\medskip

We highlight that in this section we follow closely~\cite{GaMePl1}. More precisely, we adapt the arguments developed  for networks in~\cite[Section~3.3.2 and Section~3.3.3]{GaMePl1}, to the case of one curve with endpoints constrained to the $x$-axis.
\medskip

We linearize the highest order terms of the motion equation~\eqref{motionequationparam} around the initial
parametrization $\gamma_0$ and we obtain
\begin{align}
\pat \gamma
+\frac{2}{\vert\pax \gamma_0 \vert^4}\pax^4 \gamma
&=\left(\frac{2}{\vert\pax \gamma_0 \vert^4} -\frac{2}{\vert\pax \gamma \vert^4}\right)\pax^4 \gamma 
+\widetilde{f}(\pax^3 \gamma,\pax^2 \gamma,\pax \gamma)
\\&=:f(\pax^4 \gamma,\pax^3 \gamma,
\pax^2 \gamma,\pax \gamma)\,. \label{motioneqlin}
\end{align}
Then, after noticing that the attachment condition and the second order condition are already linear, we linearize the highest order terms of the third order condition, that is
\begin{align}
\left (-\frac{1}{\vert \pax \gamma_0 \vert^3} \langle \pax^3 \gamma , \nu_0 \rangle \nu_0 \right) _1&= \left (-\frac{1}{\vert \pax \gamma_0 \vert^3} \langle \pax^3 \gamma , \nu_0 \rangle \nu_0 + \frac{1}{\vert \pax \gamma \vert^3} \langle \pax^3 \gamma , \nu \rangle \nu + h (\pax \gamma)\right) _1 
\\&=: b(\pax^3 \gamma, \pax \gamma)\, .  \label{thirdlin}
\end{align}
Thus, the linearized system associated to~\eqref{analyticprobl} is given by
\begin{equation}\label{linearprobl}
\begin{cases}
\pat \gamma+\frac{2}{\vert\pax \gamma_0 \vert^4}\pax^4 \gamma =f
\\
\gamma_2=0 \qquad &\text{attachment conditions,}\\
\pax^2 \gamma =0 &\text{second order conditions,}\\
\left(-\frac{1}{\vert \pax \gamma_0 \vert^3}
\left\langle \pax^3 \gamma,\nu_0\right\rangle \nu_0\right)_1=b
&\text{third order conditions,}\\
\gamma(0)=\gamma_0 &\text{initial condition}
\end{cases}
\end{equation}
where $f, b$ are defined in~\eqref{motioneqlin},~\eqref{thirdlin} and we have omitted the dependence on $(t,x)\in[0,T]\times[0,1]$ in the  motion
equation, on $(t,y)\in[0,T]\times\{0,1\}$ in the boundary conditions 
and on $x\in[0,1]$ in the initial condition.
\begin{rem}
   Replacing the right-hand side of system~\eqref{linearprobl} with $(f,b,\psi)$, we get the general system
    \begin{equation}\label{linearproblgeneral}
\begin{cases}
\pat \gamma+\frac{2}{\vert\pax \gamma_0 \vert^4}\pax^4 \gamma =f
\\
\gamma_2=0 \qquad &\text{attachment conditions,}\\
\pax^2 \gamma =0 &\text{second order conditions,}\\
\left(-\frac{1}{\vert \pax \gamma_0 \vert^3}
\left\langle \pax^3 \gamma,\nu_0\right\rangle \nu_0 \right)_1=b
&\text{third order conditions,}\\
\gamma(0)=\psi&\text{initial condition}
\end{cases}
\end{equation}
    where $f\in C ^{\frac\alpha4,{^{\alpha}}}([0,T]\times[0,1],\mathbb{R}^2)$, $(b(\cdot, 0),b(\cdot,1))\in C  ^{\frac{1+\alpha}{4}}([0,T], \R^2)$ and $\psi\in C^{4+\alpha}\left([0,1],\mathbb{R}^2\right)$.
\end{rem}

\begin{dfnz}\label{lincompcond}[Linear compatibility conditions]
Let $(f,b)$ be a given right-hand side to the linear system~\eqref{linearproblgeneral}. A function $\psi\in C^{4+\alpha}\left([0,1],\mathbb{R}^2\right)$ satisfies the {\em linear compatibility conditions} with respect to $(f,b)$ if for $y\in\{0,1\}$ there hold
\begin{align}
&\psi(y)_2=0 \,,
\\&\pax^2 \psi(y)=0 \,,
\\&\left(-\frac{1}{\vert \pax \gamma_0 \vert^3}\left\langle \pax^3 \psi(y),\nu_0(y) \right\rangle 
\nu_0(y) \right)_1 =b(0,y),
\\&\left(\frac{2}{\vert \pax \gamma_0 \vert^4}\pax^4\psi(y)-f(0,y)\right)_2=0 \,.
\end{align}
\end{dfnz}
\medskip

\begin{teo}\label{linearexistence}
Let $\alpha\in(0,1)$ and let $T>0$.
Suppose that
\begin{itemize}
\item $f\in C ^{\frac\alpha4,{^{\alpha}}}([0,T]\times[0,1],\mathbb{R}^2)$;
\item $(b(\cdot, 0),b(\cdot,1))\in C  ^{\frac{1+\alpha}{4}}([0,T], \R^2)\,$;
\item $\psi\in C^{4+\alpha}\left([0,1],\mathbb{R}^2\right)$;
\item $\psi$ satisfies the linear compatibility conditions in Definition~\ref{lincompcond} with respect to $(f,b)$.
\end{itemize}
Then, the 
linearized problem~\eqref{linearproblgeneral} has a unique solution $\gamma\in C ^{\frac{4+\alpha}{4},{^{4+\alpha}}}([0,T]\times[0,1],\mathbb{R}^2)$.

Moreover, for all $T>0$ there exists a $C(T)>0$ such that the solution satisfies
\begin{equation}
 \Vert \gamma\Vert_{\frac{4+\alpha}{4},^{4+\alpha}} 
\leq C(T)\left( 
 \Vert f\Vert_{\frac\alpha4,^{\alpha}}+
\Vert b\Vert_{\frac{1+\alpha}{4}}+
\Vert \psi\Vert_{4+\alpha}
\right)\, .
\end{equation} 
\end{teo}
\begin{proof}
To show the result we have to prove that
system~\eqref{linearproblgeneral} satisfies
all the hypothesis of the general~\cite[Theorem 4.9]{solonnikov1}.

Using the notation of~\cite{solonnikov1}, we write $\gamma=(u,v)$ and we denote by $b,r$, respectively, the number of boundary and initial conditions which in our case are $b=2$, $r=2$.
\\Moreover, we write the motion equation in the form
\be\label{solonn1}
\mathcal L \gamma = f
\ee
where the $2 \times 2$ matrix $\mathcal L$ is given by 
\begin{equation}
\mathcal L (x,t, \pax, \pat )=\begin{bmatrix}
\partial_t+\frac{2}{\vert \pax \gamma_0\vert^4}\partial^4_x & 0\\
0 & \partial_t+\frac{2}{\vert \pax \gamma_0\vert^4}\partial^4_x  
\end{bmatrix}
\end{equation}
and the vector $f= (f^1, f^2)$ is the right-hand side of motion equation in system~\eqref{linearproblgeneral}. 
\begin{itemize}
    \item We firstly show that system~\eqref{solonn1}  satisfies the parabolicity condition~\cite[page 8]{solonnikov1}. 
    As in~\cite{solonnikov1}, we call
$\mathcal{L}_0$ the principal part of the matrix $\mathcal L$ and we choose the integers $s_k,t_j$ in~\cite[page 8]{solonnikov1} as follows: $s_k=4$ for $k \in \{1,2\}$ and $t_j=0$ for $j \in \{1,2\}$.
Hence, we have $\mathcal{L}_0=\mathcal{L}$ and its determinant
$$
\mathrm{det}\mathcal{L}_0(x,t,i\xi,p)=
\left(  \frac{2}{\vert \pax \gamma_0 \vert}\xi^4+p\right) ^2
$$
is a polynomial of degree two in $p$ with one root 
$$p=-\frac{2}{\vert \pax \gamma_0 \vert^4}\xi^4 $$
of multiplicity two.
\\Then, choosing $\delta \leq \frac{2}{\vert \pax \gamma_0 \vert^4} $, the conditions of~\cite[page 8]{solonnikov1} are satisfied and the system is parabolic in the sense of Solonnikov. 
\item As it is shown in~\cite[pages 11-15]{eidelman2}, the compatibility condition at boundary points stated in~\cite[page 11]{solonnikov1}
is equivalent to the following Lopatinskii-Shapiro condition, which we check only
for $y=0$ (the case $y=1$ can be treated analogously).

Let $\lambda\in\mathbb{C}$ with $ \Re(\lambda)>0$ be arbitrary.
The Lopatinskii-Shapiro condition at $y$
is satisfied if every solution $\gamma\in C^4([0,\infty),\mathbb{C}^2)$ to
the system of ODEs
\begin{equation}\label{LopatinskiiShapirosystem}
\begin{cases}
\lambda \gamma(x)+\frac{1}{\vert\pax \gamma_0
\vert^4}\pax^4 \gamma(x)=0\\
\gamma(y)_2=0 \\
\pax^2 \gamma(y)=0 \\
\left(\frac{1}{\vert \pax \gamma_0 \vert ^3}
\left\langle \pax^3\gamma(y),\nu_0(y)\right\rangle \nu_0(y)\right)_1=0 \\
\end{cases}
\end{equation}where $x \in [0,\infty)$, which satisfies $\lim_{x\to\infty}\lvert \gamma(x)\rvert=0$, is the trivial solution.

To do so, we consider a solution $\gamma$ to~\eqref{LopatinskiiShapirosystem} 
such that $\lim_{x\to\infty}\lvert \gamma(x)\rvert=0$. We test the motion equation by 
$\vert\pax \gamma_0\vert\left\langle \overline{\gamma}(x),\nu_0\right\rangle \nu_0$
and we integrate twice by part to get
\begin{align}\label{afterint}
0&=\lambda\vert\pax \gamma_0\vert \int_0^\infty\vert \left\langle\gamma(x), \nu_0\right\rangle\vert^2
\,\mathrm{d}x
+\frac{1}{\vert\pax \gamma_0\vert^3} 
\int_0^\infty\vert \left\langle \pax^2 \gamma(x),\nu_0\right\rangle\vert^2\,\mathrm{d}x \nonumber\\
&+\frac{1}{\vert\pax \gamma_0\vert^3}\left\langle\overline{\gamma}(0),\nu_0\right\rangle \left\langle \pax^3 \gamma(0),\nu_0 \right\rangle 
-\frac{1}{\vert\pax \gamma_0\vert^3}
\left\langle\overline{\pax \gamma}(0),\nu_0\right\rangle \left\langle \pax^2 \gamma(0),\nu_0 \right\rangle \,,
\end{align}
where we have already used the fact that all derivatives decay to zero for $x$ tending to
infinity, due to the specific exponential form of the
solutions to~\eqref{LopatinskiiShapirosystem}.
We now observe that, since $\gamma_0$ is an admissible initial parametrization, 
the first component of $\nu_0$ is bounded from below. That is, from the third order condition in system~\eqref{LopatinskiiShapirosystem} it follows that $\left\langle \pax^3 \gamma(0),\nu_0 \right\rangle=0$. Thus, this condition together with the second order condition implies that the boundary terms in~\eqref{afterint} vanish. Then, taking the real part of~\eqref{afterint} and recalling that $ \Re(\lambda)>0$, we have $\left\langle \gamma(x),\nu_0 \right\rangle=0$ for all $x \in [0, \infty)$. In particular, from the attachment condition in~\eqref{LopatinskiiShapirosystem}, it follows that $\gamma(0)=0$.
\\As before, testing the motion equation by $\vert\pax \gamma_0\vert\left\langle \overline{\gamma}(x),\tau_0\right\rangle \tau_0$ and integrating by part, we get 
\begin{align}\label{afterint2}
0&=\lambda\vert\pax \gamma_0\vert \int_0^\infty\vert \left\langle\gamma(x), \tau_0\right\rangle\vert^2
\,\mathrm{d}x
+\frac{1}{\vert\pax \gamma_0\vert^3} 
\int_0^\infty\vert \left\langle \pax^2 \gamma(x),\tau_0\right\rangle\vert^2\,\mathrm{d}x\nonumber\\
&+\frac{1}{\vert\pax \gamma_0\vert^3}\left\langle\overline{\gamma}(0),\tau_0\right\rangle \left\langle \pax^3 \gamma(0),\tau_0 \right\rangle 
-\frac{1}{\vert\pax \gamma_0\vert^3}
\left\langle\overline{\pax \gamma}(0),\tau_0\right\rangle \left\langle \pax^2 \gamma(0),\tau_0 \right\rangle \,.
\end{align}
The boundary term in~\eqref{afterint2} vanishes since $\gamma(0)=0$ and the second order condition holds. Hence, considering again the real part of~\eqref{afterint2} we have that $\left\langle\gamma(x), \tau_0\right\rangle=0$ for all $x \in [0,\infty).$ So, we conclude that $\gamma(x)=0$ for all $x \in [0,\infty)$. 

\item Finally, to check the complementary condition for the initial datum stated in~\cite[page 12]{solonnikov1},
we observe that the $2\times 2$ matrix $[C_{\alpha j}]$ is the identity matrix.
Then, choosing $\gamma_{\alpha j}=0$ for $\alpha\in\{1,2\}$ and $j\in\{1,2\}$,
we obtain $\rho_\alpha=0$ and $C_0=Id$.
\\Moreover, the rows of the matrix $\mathcal{D}(x,p)=\hat{\mathcal{L}_0}(x,0,0,p)=p Id$
are linearly independent modulo the polynomial $p^2$.
\end{itemize}
\end{proof}

\subsection{Short-time existence of the analytic problem}
From now on, we fix $\alpha\in (0,1)$ and we consider an admissible initial parametrization $\gamma_0$ as in Definition~\ref{admissinitialanalytic}, 
with $\Vert \gamma_0 \Vert_{4+\alpha}=R$. Moreover, with a slight abuse of notation, we denote by $b(\cdot)$ the vector $(b( \cdot, 0), b(\cdot, 1))$ in the statement of Theorem~\eqref{linearexistence}.

\begin{dfnz}
For $T>0$
we define the linear spaces
\begin{align*}
\mathbb{E}_T:=\{&\gamma\in 
C  ^{\frac{4+\alpha}{4},{^{4+\alpha}}}([0,T]\times[0,1],\mathbb{R}^2)
\;\text{such that for}\;t\in[0,T]\,,
\\&\text{attachment and second order conditions hold}\}
\,,\\
\mathbb{F}_T:=\{&(f,b,\psi)\in 
C  ^{\frac{\alpha}{4},{^{\alpha}}}([0,T]\times[0,1],\mathbb{R}^2)
\times 
C  ^{\frac{1+\alpha}{4}}([0,T], \R^2)
\times  C^{4+\alpha}\left([0,1],\mathbb{R}^2\right)\\
&\text{such that the linear compatibility conditions hold}
\}
\,,
\end{align*}
endowed with the norms
\begin{align}
   \Vert \gamma \Vert_{\mathbb{E}_T}&=  \Vert \gamma \Vert_{\frac{4+\alpha}{4},4+\alpha} \, ,
   \\\Vert (f,b,\psi) \Vert_{\mathbb{F}_T}& = \Vert f \Vert_{\frac{\alpha}{4}, \alpha}+ \Vert b \Vert_{\frac{1+\alpha}{4}}+ \Vert \psi \Vert_{4+\alpha}.
\end{align}
Moreover, we consider the affine spaces
\begin{align*}
\mathbb{E}^0_T:=\{&\gamma\in 
\mathbb{E}_T\,\text{such that }\,\gamma_{\vert t=0}=\gamma_0\}
\,,\\
\mathbb{F}^0_T:=\{&(f,b)\,\text{such that }\, (f,b,\gamma_0)\in\mathbb{F}_{T}\}
\times\{\gamma_0\}
\,.
\end{align*}
\end{dfnz}
We remark that Lemma~\ref{3.17} and Lemma~\ref{3.23} below are respectively {\cite[Lemma~3.17]{GaMePl1}} and {\cite[Lemma~3.23]{GaMePl1}}.
\begin{lemma}\label{3.17}
For $T>0$, the map $L_{T}:\mathbb{E}_T\to \mathbb{F}_T$ defined by
$$
L_{T}(\gamma):=
\begin{pmatrix}
\pat\gamma+\frac{2}{\vert\pax \gamma_0\vert^4}\pax^4 \gamma\\
\left(-\frac{1}{\vert \pax \gamma_0 \vert^3}
\left\langle \pax^3\gamma,\nu_0\right\rangle \nu_0 \right)_1\\
\gamma_0 
\end{pmatrix} \, ,
$$
is a continuous isomorphism. 
\end{lemma}
\noindent In the following 
we denote by $L^{-1}_T$ the inverse of $L_T$, by $B_M$ the open ball of radius $M>0$ and center $0$ in $\mathbb{E}_T$ and by $\overline {B_M}$ its closure.

Before proceeding we notice that, since the admissible initial parametrization $\gamma_0 : [0,1] \to \R^2$ is a regular curve, there exists a constant $C>0$ such that
\be\label{est0}
\inf_{x\in[0,1]}\vert\pax \gamma_0 \vert \geq C\,,
\ee
which obviously implies that
$$
\sup_{x\in[0,1]}\frac{1}{\vert \pax \gamma_0 \vert} \leq \frac{1}{C}\,.
$$
Then, as it is shown in~\cite{GaMePl1}, there exists a constant $\widetilde{C}$ depending on $R$ and $C$,  such that for every $j\in\mathbb{N}$ it holds
$$
\left\lVert \frac{1}{\vert \pax \gamma_0\vert^j}
\right\rVert_{\alpha}
\leq \left(\frac{\Vert \pax \gamma_0 \Vert_{\alpha}}{C^2}\right)^j
\leq \left(\frac{R}{C^2}\right)^j\qquad \text{and} \qquad
\left\lVert \frac{1}{\vert \pax \gamma_0 \vert^j}\right\rVert_{1+\alpha}
\leq \widetilde{C}(R,C)\, .
$$
We also notice that these estimates are preserved during the flow. More precisely, following the proof in~\cite{GaMePl1}, one can show that there exists $\widetilde{T}(M, C)\in(0,1]$ such that 
for $T\in[0,\widetilde{T}(M,C)]$ every curve $\gamma\in \mathbb{E}^0_T\cap B_M$ is regular 
and for all $t\in[0,\widetilde{T}(M,C)]$ it holds
\be\label{inverseest}
\sup_{x\in[0,1]} \frac{1}{\vert\pax \gamma(t,x)\vert}\leq\frac{2}{C}\,.
\ee
Furthermore, for every $j \in \N$ and $y \in \{0,1\}$, we have
\be\label{jestimate}
\left\lVert \frac{1}{\vert \pax \gamma\vert^j}\right\rVert_{\frac{\alpha}{4},\alpha}
\leq \left(\frac{4M}{C^2}\right)^j \qquad \text{and} \qquad \left\lVert \frac{1}{\vert \pax \gamma (y)\vert^j}\right\rVert_{\frac{1+\alpha}{4}}
\leq \widetilde{C}(R,C)\,.
\ee

\begin{lemma}\label{3.23}
    For $T\in (0, \widetilde{T}(M, C)]$, the map $N_T(\gamma):= (N_{T,1}, N_{T,2}, \gamma_0)$ given by 
    \begin{align*}
N_{T,1}:&
\begin{cases}
\mathbb{E}^0_T &\to C  ^{\frac\alpha4,
{^{\alpha}}}([0,T]\times[0,1],\mathbb{R}^2),\\
\gamma&\mapsto
f(\gamma):=f(\pax^4 \gamma, \pax^3 \gamma, \pax^2 \gamma, \pax \gamma),
\end{cases}\\
N_{T,2}:&
\begin{cases}
\mathbb{E}^0_T&\to  C^{\frac{1+\alpha}{4}}([0,T],\mathbb{R}^2), \\
\gamma&\mapsto b(\gamma):= b(\pax^3 \gamma, \pax \gamma)
\end{cases}
\end{align*}
where $f,b$ are defined in~\eqref{motioneqlin},~\eqref{thirdlin} respectively, is a well defined mapping from $\mathbb{E}^0_T$ to $\mathbb{F}^0_T$.
\end{lemma}
\begin{proof}
    We have that $\gamma(t,\cdot)$ is a regular curve thanks to the discussion above, hence $N_T$ is well defined.
    In order to show that $N_T(\gamma) \in \mathbb{F}^0_T$, we have to prove that $\gamma_0$ satisfies the linear compatibility conditions with respect to $(N_{T,1}, N_{T,2})$. This easily follows from the definition of $N_{T,1},N_{T,2}$ and the fact that $\gamma_0$ is an admissible initial parametrization as in Definition~\ref{admissinitialanalytic}.
\end{proof}
\begin{dfnz}
  Let $\gamma_0$ be an admissible initial parametrization and let  $C>0$ the constant given by~\eqref{est0}. For $M>0$ and $T \in (0, \widetilde{T}(M, C)$ we define the mapping $K_T:\mathbb{E}^0_T \to  \mathbb{E}^0_T$ as $K_T:= L^{-1}_T N_T$.
\end{dfnz}
With a proof similar to~\cite[Proposition~3.28 and Proposition~3.29]{GaMePl1} one can prove the following result.
\begin{prop}\label{welldefined}
There exists a positive radius $M=M(R,C)$ and a positive time $\hat{T}(M) \in (0, \widetilde{T}(M,C))$ 
such that for all $T\in (0,\hat{T}(M)]$
the map
$K_T:\mathbb{E}^0_T\cap \overline{B_{M}}\to\mathbb{E}^0_T\cap \overline{B_{M}}$ 
is well-defined and it is a contraction.
\end{prop}
\begin{teo}\label{existenceanalyticprob}
Let $\gamma_0$ be an admissible initial parametrization as in Definition~\ref{admissinitialanalytic}. There exists a positive radius $M$ and a positive time $T$ such that the system~\eqref{analyticprobl}
has a unique solution in $C^{\frac{4+\alpha}{4},4+\alpha}\left([0,T]\times[0,1]\right)\cap \overline{B_M}$.	
\end{teo}
\begin{proof}
Let $M$ and $\hat{T}(M)$ be the radius and time as in Proposition~\ref{welldefined} and let $T\in (0,\hat{T}(M)]$. The solutions of~\eqref{analyticprobl} in $C^{\frac{4+\alpha}{4},4+\alpha}\left([0,T]\times[0,1]\right)\cap \overline{B_M}$ are the fixed points of $K_T$ in $\mathbb{E}_T^0\cap \overline{B_M}$. Moreover, it is unique by the Banach-Caccioppoli contraction theorem as $K_T$ is a contraction of the complete metric space $\mathbb{E}_T^0\cap\overline{B_M}$.
\end{proof}
\subsection{Geometric existence and uniqueness}
In Theorem~\ref{existenceanalyticprob} we show that there exists a unique solution to the analytic problem~\eqref{analyticprobl} provided that the initial curve is admissible. 
In this section, we first establish a relation between geometrically admissible initial curves and admissible initial parametrizations, then we show the geometric uniqueness of the flow, in the sense that up to reparametrization the geometric problem~\eqref{evolutionlaw} has a unique solution.

We remark that the following technique was introduced by Garcke and Novick-Cohen in~\cite{GarcNovic}, and then it has been employed, for instance, by Garke, Pluda at al. in~\cite{GoMePlu,GaMePl1, GaMePl2} for the case of shortening and elastic flows of networks.
\begin{lemma}\label{repara}
Suppose that $\gamma_0$ is a geometrically admissible initial curve as in Definition~\ref{admissinitialgeom}. Then, there exists a smooth function $\psi_0:[0,1]\to [0,1]$ such that the reparametrization 
$\widetilde \gamma_0 = \gamma_0 \circ \psi_0$
of  $\gamma_0$ is an admissible initial parametrization for the analytic problem~\eqref{analyticprobl}.
\end{lemma}
\begin{proof}
We look for a smooth map $\psi_0:[0,1]\to [0,1]$ with 
 $\pax \psi_0 (x)\neq 0$ for every $x\in [0,1]$, such that $\widetilde \gamma_0 = \gamma_0 \circ \psi_0 : [0,1] \to \R^2$ is regular and
of class $C^{4+\alpha}([0,1])$. If $\psi_0(y)=y$ for $y \in \{0,1\}$, then $\widetilde \gamma_0$ clearly satisfies the attachment condition. Moreover, since the geometric quantities are invariant under reparametrization, also the non-degeneracy condition and the third-order condition are still satisfied.
In order to fulfil the second order condition $\pax^2 \widetilde \gamma_0(y)=0$, 
we consider a map $\psi_0$ such that 
\be 
\pax \psi_0(y)=1 \qquad \text{and} \qquad 
 \pax^2 \psi_0 (y)=-\frac{\pax^2 \gamma_0 (y)}{\pax \gamma_0(y)}
 \ee
 for $y\in\{0,1\}$. Thus, it remains to show that 
 \be
 \left ( \widetilde V_0 \widetilde \nu_0 + \widetilde T_0 \widetilde \tau_0 \right)_2 =0 \, .
 \ee
 As we notice above, this is equivalent to 
  \be
 \left ( V_0  \nu_0 + \widetilde T_0 \tau_0 \right)_2 =0 \, ,
 \ee
 however, since $\gamma_0$ is a geometrically admissible initial curve, it is enough to prove that
 \be\label{4ordcond}
 \widetilde T_0 - T_0=0 \, .
 \ee
 Thus, asking that $\pax^3 \psi_0(y) =1$, we rewrite relation~\eqref{4ordcond} as
 \be
 g_1 (\pax \gamma ) (y) \pax^4 \psi_0 (y) + g_2 (\pax \gamma, \pa^2 \gamma, \pax^3 \gamma)(y)=0
 \ee
 where $g_1, g_2$ are non-linear functions. Hence, $\pax^4 \psi_0(y)$ are uniquely determined
for $y \in \{0,1\}$. In the end, we may choose $\psi_0$ to be the fourth Taylor polynomial near each boundary point, 
join these values up inside the interval $(0,1)$ and then make it smooth. 
\end{proof}

\begin{dfnz}\label{maxsol}
    Let $\gamma_0$ be a geometrically admissible initial curve as in Definition~\ref{admissinitialgeom} and $T>0$. A time-dependent family of curves $\gamma_t$ for $t \in [0,T)$ is a maximal solution to the elastic flow with initial datum $\gamma_0$, if it is a solution in the sense of Definition~\ref{Def:elasticflow} in $[0, \hat T]$ for some $\hat T < T$ and if there does not exist a solution $\widetilde \gamma_t$ in $[0, \widetilde T]$  with $\widetilde T > T$ and such that $\gamma=\widetilde \gamma$ in $(0,T)$.
\end{dfnz}
Following the arguments in~\cite[Lemma~5.8 and Lemma~5.9]{GaMePl2}, one can show that a maximal solution to the elastic flow always exists and it is unique up to reparametrization. Hence, from now on we only consider the time $T$ in Definition~\ref{maxsol}, which we call maximal time of existence and we denote by $T_{\max}$.
 \medskip \\We notice that the following theorem is slightly different to the corresponding one in~\cite{GaMePl2}, where the authors firstly prove the geometric uniqueness in a ``generic'' time interval $[0,T]$ and then they show the existence of $T_{\max}$ using the fact that the solution is unique in a geometric sense. However, with an intermediate step, the result can be stated as follows.
\begin{teo}\label{geomexistence}[Geometric existence and uniqueness]
Let $\gamma_0$ be a geometrically admissible initial curve as in Definition~\ref{admissinitialgeom}. Then, there exists
a positive time $T_{\max}$ such that within the time interval $[0,T_{\max})$ there is a unique elastic flow $\gamma_t$ in the sense of Definition~\ref{Def:elasticflow}.
\end{teo}
\begin{proof}
By Lemma~\ref{repara} there exists a reparametrization $\widetilde \gamma_0$ of $\gamma_0$ which is an admissible initial parametrization in the sense of Definition~\ref{admissinitialanalytic}. Then, by Theorem~\ref{existenceanalyticprob} there exists a solution $\widetilde \gamma_t$ of system~\eqref{analyticprobl} in some maximal time interval $[0, \widetilde T_{\max}]$. In particular, $\widetilde \gamma_t$ is a solution to system~\eqref{evolutionlaw}. 
\\Let us suppose that $\gamma_t$ is another solution to the elastic flow in sense of Definition~\ref{Def:elasticflow} in a time interval $[0, T']$, with the same geometrically admissible initial curve. 
We aim to show that there exists a time $T_{\max} \in (0, \max \{\widetilde T_{\max}, T'\})$ such that $\widetilde \gamma_t = \gamma_t$ (as curves) for every $t \in [0, T_{\max}]$.
\\To be precise, we need to construct a regular reparametrization $\psi(t,x): [0, T_{\max}] \times [0,1] \to [0,1]$, such that the reparametrized curve $\sigma (t,x) = \gamma(t, \psi(t,x))$ is a solution to the analytic problem~\eqref{analyticprobl} and coincides with $\widetilde \gamma_t$ in a possibly small but positive time interval.
Hence, computing the space and time derivatives of $\sigma(t,x)$ as a composed function and replacing in the evolution equation 
\be
     \pat \sigma (t,x) = \frac{\pax^4 \sigma }{\vert \pax \sigma \vert ^4}+ l.o.t.\label{eqsigma}
     \ee
     we get the following evolution equation for $\psi$
      \begin{align}
         \pat \psi (t,x) = & - \frac{ \langle \pat \gamma (t, \psi(t,x)) ,\pax \gamma (t, \psi(t,x)) \rangle }{\vert \pax \gamma (t, \psi(t,x)) \vert ^2} +\frac{\langle \pax^4 \gamma (t, \psi(t,x)),\pax \gamma (t, \psi(t,x)) \rangle  }{\vert \pax \gamma(t, \psi(t,x) \vert ^6}
         \\+ & \frac{6 \langle \pax^3 \gamma (t, \psi(t,x)), \pax \gamma (t, \psi(t,x)) \rangle \pax^2 \psi (t,x)}{\vert \pax \gamma(t, \psi(t,x) \vert ^6 (\pax \psi(t,x))^2}
          +\frac{3 \langle \pax^2 \gamma (t, \psi(t,x)), \pax \gamma (t, \psi(t,x)) \rangle (\pax^2 \psi (t,x))^2}{\vert \pax \gamma(t, \psi(t,x) \vert ^6 (\pax \psi(t,x))^4}
          \\&+\frac{ 4 \langle \pax^2 \gamma (t, \psi(t,x)), \pax \gamma (t, \psi(t,x)) \rangle \pax^3 \psi (t,x)}{\vert \pax \gamma(t, \psi(t,x) \vert ^6 (\pax \psi(t,x))^3}
      + \frac{\pax^4 \psi (t,x)} {\vert \pax \gamma(t, \psi(t,x) \vert ^2 (\pax \psi(t,x))^4}+ l.o.t. \, .
     \end{align}
     Taking into account the boundary conditions, we have that such parametrization has to satisfy the following boundary value problem
     \be
     \begin{cases}
         \pat \psi (t,x) =\frac{\pax^4 \psi (t,x)} {\vert \pax \gamma(t, \psi(t,x) \vert ^2 (\pax \psi(t,x))^4} + g
         \\ \psi (t, y)=y
         \\\pax^2 \psi (t,y) = - \frac{\langle \pax^2 \gamma(t, \psi(t,x)), \pax \gamma (t, \psi(t,x)) \rangle (\pax \psi)^2}{ \vert \pax \gamma (t, \psi(t,x))\vert^2}
         \\ \psi(0,x)=\psi_0(x)
     \end{cases}\label{psiproblem}
     \ee
     for $y \in \{0,1\}$ and $t \in [0, T_{\max}]$, where the function $\psi_0$ is given by Lemma~\ref{repara} and the terms in $g$ depend on the solution $\psi$, $\pax ^j \psi$ for $j \in \{1,2,3\}$ and $\pat \gamma$, $\pax^j \gamma$ for $j \in \{1,2,3,4\}$. From the computation above, it follows that the function $\gamma$ and its time-space derivatives depend also on $\psi$. To remove this dependence, we consider the associated problem for the inverse of $\psi$, that is $\xi (t, \cdot ) = \psi^{-1} (t, \cdot)$. So, the differentiation rules
\begin{align}
\paz \xi (t,z) = &\pax \psi (t, \xi (t, z))^{-1}
\\\paz^2 \xi (t,z)=& - (\paz \xi(t,z))^3 \pax^2 \psi(t,\xi(t,z))
\\\paz^3 \xi (t,z)=& 3 \frac{(\paz^2 \xi(t,z))^2}{\paz \xi(t,z)}- (\paz \xi(t,z))^4 \pax^3 \psi(t,\xi(t,z))
\\\paz^4 \xi (t,z)=&- 15 \frac{ (\paz^2 \xi(t,x))^3}{(\paz \xi(t,x))^2}
 +10\frac{\paz^2 \xi (t,z) \paz^3 \xi (t,z)}{\paz \xi(t,z)}- (\paz \xi (t,z))^5\pax^4 \psi(t, \xi(t,z))
  \end{align}
yield the evolution equation  \begin{align} 
         \pat \xi (t,z) = & - \frac{ \langle \pat \sigma (t, z) ,\paz \sigma (t, z) \rangle }{\vert \paz \sigma (t, z) \vert ^2} \paz \xi (t,z) +\frac{\langle \paz^4 \sigma (t, z),\paz \sigma (t, z) \rangle  }{\vert \paz \sigma(t, z) \vert ^6} \paz \xi (t,z)
         \\- & \frac{6 \langle \paz^3 \sigma (t, z), \paz \sigma (t, z) \rangle }{\vert \paz \sigma(t, z) \vert ^6 } \paz^2 \xi (t,z)
          +\frac{3 \langle \paz^2 \sigma (t, z), \paz \sigma (t, z) \rangle }{\vert \paz \sigma(t, z) \vert ^6 } \frac{(\paz^2 \xi (t,z))^2}{\paz \xi (t,z)}
          \\&+\frac{ \langle \paz^2 \sigma (t, z), \paz \sigma (t, z) \rangle}{\vert \paz \sigma(t, z) \vert ^6 } \left ( -4 \paz^3 \xi (t,z) + \frac{12 (\paz^2\xi (t,z))^2}{\paz \xi (t,z)}\right)
      \\&+ \frac{1} {\vert \paz \sigma(t, z) \vert ^2 } \left( - \paz^4 \xi (t,z) +\frac{10 \paz^2 \xi (t,z) \paz^3 \xi (t,z)}{\paz \xi (t,z)} - \frac{15(\paz^2 \xi (t,z))^3}{(\paz \xi (t,z))^2}\right)+ l.o.t. \, .
  \end{align}
  Hence, we obtain the following system for $\xi$
  \be
  \begin{cases}
         \pat \xi (t,z) =-\frac{ \paz^4 \xi (t,z)} {\vert \paz \sigma(t, z )\vert ^2 } +g
         \\ \xi(t,y)= y
         \\ \paz^2 \xi (t,y)= \frac{\langle \paz^2 \sigma (t,y), \paz \sigma (t,y) \rangle \paz \xi (t,y)}{\vert \paz \sigma(t,y) \vert^2}
         \\ \xi(0,z)= \psi_0^{-1}(z)
  \end{cases}\label{xiproblem}
  \ee
  where $g$ is a non-linear smooth function which depends on $\pax^j \xi$ for $j \in \{1,2,3\}$, $\paz^ \sigma$ for $j \in \{1, 2,3,4\}$, $\pat \sigma$. We now observe that the system~\eqref{xiproblem} has a very similar structure as~\eqref{linearproblgeneral}, hence, after linearize, we apply the linear theory developed by Solonnikov in~\cite{solonnikov1} and we get well–posedness. Contraction estimates allow us to conclude the existence and uniqueness of solution with a fixed-point argument. Reversing the above argumentation, we obtain that the function $\psi$ solves system~\eqref{psiproblem}.
\\Then, $\sigma_t$ is a solution to system~\eqref{analyticprobl}. Indeed, the motion equation follows from~\eqref{psiproblem} and the geometric evolution of $\gamma_t$ in normal direction. The geometric boundary conditions, namely attachment, curvature, and third-order conditions, are satisfied as $\gamma_t$ is a solution to the geometric problem. Moreover, the boundary conditions in system~\eqref{psiproblem} ensure that $\sigma_t$ satisfies the second order condition. 
\\Thus, by uniqueness of the analytic problem proved in Theorem~\ref{existenceanalyticprob}, $\sigma_t$ (that is $\gamma_t$ up to reparametrization) and $\widetilde \gamma_t$ need to coincide on a possibly small time interval.
\end{proof}

\section{Curvature bounds}
To simplify the notation, we introduce the following polynomials.
\begin{dfnz}\label{polinomi}
For $h \in \N$, we denote by $\pol_\sigma^h(k)$
a polynomial in $k,\dots,\pas^h k$ with constant 
coefficients in $\mathbb{R}$ such that every monomial it contains is of the form
\begin{equation*}
C \prod_{l=0}^h	(\pas^lk)^{\alpha_l}\quad\text{ with} \quad \sum_{l=0}^h(l+1)\alpha_l = \sigma\,,
\end{equation*}
where $\alpha_l\in\mathbb{N}$ for $l\in\{0,\dots,h\}$ and $\alpha_{l_0}\ge 1$ 
for at least one index $l_0$.
\end{dfnz}
\begin{rem}\label{prop-polinomi}
One can easily prove that
\begin{align}
\partial_s\left(\pol_\sigma^h( k)\right)&=\pol_{\sigma+1}^{h+1}( k)\,,\nonumber\\
\mathfrak{p}_{\sigma_1}^{h_1}(k)\mathfrak{p}_{\sigma_2}^{h_2}
(k)&=\mathfrak{p}_{\sigma_1+\sigma_2}^{\max\{h_1,h_2\}}(k)\,,
\\
\mathfrak{p}_\sigma^{h_1}(k)+ \mathfrak{p}_\sigma^{h_2}(k) &= \mathfrak{p}_\sigma^{\max\{h_1,h_2\}}(k). \label{calcpol} 
\end{align}
Moreover, following the arguments in~\cite{ManPluPozSurvey}, it holds
\be\label{dtpolinomi}
\pat \left ( \pol_{\sigma}^h (k) \right)  = \pol_{\sigma+4}^{h+4} (k)+ \Lambda \pol_{\sigma+1}^{h+1} (k)+ \mu \pol_{\sigma+2}^{h+2} (k) \, .
\ee
\end{rem} 
\begin{lemma}\label{evoluzionigeom2}
If $\gamma$ satisfies~\eqref{motionequationtang}, 
then for any $j \in \N$ the 
$j$-th
derivative of scalar curvature of $\gamma$ satisfies
\begin{align}
\partial_t\partial_s^j k
&=-2\partial_{s}^{j+4}k	
-5k^2\partial_s^{j+2}k
+\mu\,\partial_s^{j+2}k
+\Lambda\partial_{s}^{j+1}k
+\mathfrak{p}_{j+5}^{j+1}\left(k\right)+
\mu\,\mathfrak{p}_{j+3}^{j}(k) \, .
\label{derivativekt}
\end{align}
\end{lemma}
\begin{proof}
For $j=0$ we have
\begin{align}\pat k &=-2\partial_{s}^{4}k-5k^{2}\partial_{s}^{2}k-6k\left(\partial_{s}k\right)^{2}
+\Lambda\partial_{s}k-k^{5}+\mu \left( \partial_{s}^{2}k+k^3\right)
\\&= -2\partial_{s}^{4}k-5k^{2}\partial_{s}^{2}k + \mu \pas^2 k + \Lambda \pas k + \pol_5^1(k) + \mu \pol_4^0(k)\,.
\end{align} 
Then, assuming that relation~\eqref{derivativekt} is true for $j$, we show that 
\begin{align}
    \pat \pas^{j+1} k =& \pat \pas \pas^{j} k = \pas \pat \pas^{j} k + (kV-\pas \Lambda) \pas^{j+1} k
    \\=& -2 \pas^{j+5} k -10 k \pas k \pas^{j+2}k -5 k^2 \pas^{j+3} k + \mu \pas^{j+3} k + \pas \Lambda \pas^{j+1} k + \Lambda \pas^{j+2}k 
    \\&+ \pol_{j+6}^{j+2} + \mu \pol_{j+4}^{j+1} -2k \pas^2 k \pas^{j+1}k - k^4 \pas^{j+1}k + \mu k^2 \pas^{j+1}k - \pas \Lambda \pas^{j+1}k
    \\=& -2\partial_{s}^{j+5}k	
-5k^2\partial_s^{j+3}k
+\mu\,\partial_s^{j+3}k
+\Lambda\partial_{s}^{j+2}k
+\pol_{j+6}^{j+2}\left(k\right)+
\mu\,\pol_{j+4}^{j+1}(k)\, .
\end{align}
By induction, formula~\eqref{derivativekt} holds for any $j \in \N$.
\end{proof}
\subsection{Bound on $\Vert{\pas^2 k }\Vert_{L^2}$}
We aim to show that, once the following condition is satisfied, the tangential velocity behaves as the normal velocity at boundary points.
\begin{dfnz}\label{defunifnondeg}
   Let $\gamma_t$ be a maximal solution to the elastic flow in $[0,T_{\max})$. We say that $\gamma_t$ satisfies the {\em uniform non-degeneracy condition} if there exists $\rho>0$ such that
   \be\label{unifconditiontau}
    \tau_2(y) \geq \rho
    \ee
    for every $t \in[0,T_{\max})$ and $y \in \{0,1\}$.
\end{dfnz}
\begin{lemma}\label{tangent}
    Let $\gamma_t$ be a maximal solution to the elastic flow of curves subjected to boundary conditions~\eqref{navierbc}, such that the uniform non-degeneracy condition~\eqref{unifconditiontau} holds in $[0,T_{\max})$.
    Then, for every $t \in[0,T_{\max})$ and $y \in \{0,1\}$, the tangential velocity is proportional to the normal velocity, that is
    \be
    \Lambda(y)\approx \pas^2 k (y) \, .\ee
\end{lemma}
\begin{proof}
Since the boundary points are constrained to the $x$-axis, we have that
\be
(\pat \gamma_t)_2 (y)=-2 \pas^2 k (y) \nu_2 (y)+\Lambda(y) \tau_2(y)=0
\ee
for $y \in \{0,1\}$ and $t \in [0,T_{\max})$. By the fact that $\tau_2$ (hence, $\nu_2$) are bounded from below at boundary points, 
it follows
\be
\Lambda(y) = 2 \pas^2 k (y) \dfrac{\nu_2(y)}{\tau_2(y)} \approx \pas^2 k (y) 
\ee
for $t \in [0,T_{\max})$ and $y \in \{0,1\}$.
\end{proof}

\begin{prop}\label{propNavierj=2}
Let $\gamma_t$ be a maximal solution to the elastic flow of curves subjected to boundary conditions~\eqref{navierbc} with initial datum
$\gamma_0$, which satisfies the {\em uniform non-degeneracy condition}~\eqref{defunifnondeg}  in the maximal time interval $[0,T_{\max})$.
Then, for all $t\in [0,T_{\max})$, it holds
\begin{equation}\label{stimaNavierj=2}
\frac{d}{dt}\int_{\gamma} \vert\pas^2 k\vert^2\de  s
\leq C ( \mathcal{E}(\gamma_0)) \, .
\end{equation}
\end{prop}
\begin{proof}
From formula~\eqref{derivativekt} we have
    \begin{align}
   \dfrac{d}{dt} \int_\gamma\vert\pas^2 k\vert^2\de  s = & \int_\gamma 2 \pas^2 k \pat \pas^2 k + (\pas^2 k )^2 (\pas \Lambda - kV)\de  s 
\\= & \int_\gamma - 4 \pas^2 k \pas^6 k - 10 k^2 \pas^2 k \pas^4 k +2 \mu \pas^2 k \pas^4 k + 2 \Lambda \pas^3 k \pas^2 k 
\\&\phantom{\int_\gamma }+ \pol_{10}^3 (k) + \mu \pol_8 ^2 (k) + (\pas^2 k )^2 (\pas \Lambda - kV) \, \de s
\\= & \int_\gamma - 4 \pas^2 k \pas^6 k - 10 k^2 \pas^2 k \pas^4 k +2 \mu \pas^2 k \pas^4 k \\&\phantom{\int_\gamma }+ 2 \Lambda \pas^3 k \pas^2 k +2 \pas \Lambda (\pas^2 k)^2
+ \pol_{10}^3 (k) + \mu \pol_8 ^2 (k) \de s\, .
\end{align}
Thus, the terms involving the tangential velocity can be written as
\be\label{intj=2terminiconT}
 \int_\gamma \pas \Lambda (\pas^2 k )^2 + 2 \Lambda \pas^2 k \pas^3 k \de  s =  \int_\gamma \pas ( \Lambda (\pas^2 k )^2)  \de  s=  \Lambda (\pas^2 k )^2 \Big \vert_0 ^1 \, .
\ee
Moreover, integrating by parts the other terms, we get
\begin{align}
   \dfrac{d}{dt} \int_\gamma\vert\pas^2 k\vert^2\de  s=&
\int_\gamma -4(\pas^4 k)^2  - 2 \mu(\pas^3k)^2 + \pol_{10}^3 (k) + \mu \pol_8 ^2 (k) \de s
 \\& +  \Lambda (\pas^2 k )^2 \Big \vert_0 ^1 +4 ( \pas^3 k \pas^4 k- \pas^2 k \pas^5 k )\Big \vert_0 ^1  -10 k^2 \pas^2 k \pas^3 k \Big \vert_0 ^1 
 \\&- 12 (\pas k )^3 \pas^2 k \Big \vert_0 ^1+ 2 \mu \pas^2 k \pas^3 k \Big \vert_0 ^1 \, .\label{intj=2}
\end{align}
Using Navier boundary conditions, the boundary terms in equation~\eqref{intj=2} reduce to
\be\label{intj=2Navier}
  \Lambda (\pas^2 k )^2 \Big \vert_0 ^1 +4 ( \pas^3 k \pas^4 k- \pas^2 k \pas^5 k )\Big \vert_0 ^1   - 12 (\pas k )^3 \pas^2 k \Big \vert_0 ^1+ 2 \mu \pas^2 k \pas^3 k \Big \vert_0 ^1 \, .
\ee
We aim to lower the order of the second and third terms in~\eqref{intj=2Navier}. In particular, differentiating in time the condition $k(y)=0$ using relation~\eqref{kt}, we have
\be\label{j=2Navier1}
4 \pas^3 k \pas^4 k= 2 \Lambda \pas k \pas^3 k+2 \mu \pas^2 k \pas^3 k \, .
\ee
From conditions in~\eqref{navierbc}, it follows
\be
\pat \langle \gamma, 2 \pas k \nu - \mu \tau \rangle = \langle V \nu + \Lambda \tau, \pat(2 \pas k \nu - \mu \tau )\rangle=0\, ,
\ee
then, computing the scalar production using~\eqref{derivativekt}, we obtain
\begin{align}
 0=&-2 \pat \pas k V + 2 \Lambda \pas k \pas V + \mu V \pas V
\\= & 4 \pat \pas k \pas^2 k + 2 \Lambda \pas k (-2 \pas^3 k + \mu \pas k) -2\mu \pas^2 k (-2 \pas^3 k + \mu \pas k)
\\= & 4 \pas \pat k \pas^2 k- 4 \pas \Lambda\pas k \pas^2 k + 2 \Lambda \pas k (-2 \pas^3 k + \mu \pas k) 
\\&-2\mu \pas^2 k (-2 \pas^3 k + \mu \pas k)
\\= & -8 \pas^2 k \pas^5k  - 24 (\pas k)^3 \pas^2 k + 4\pas \Lambda \pas k \pas^2 k+ 4\Lambda (\pas^2 k)^2 
\\&+4\mu \pas ^2 k \pas^3 k- 4 \pas \Lambda\pas k \pas^2 k
\\&+ 2 \Lambda \pas k (-2 \pas^3 k + \mu \pas k) -2\mu \pas^2 k (-2 \pas^3 k + \mu \pas k)
\\=  & -8 \pas^2 k \pas^5k  - 24 (\pas k)^3 \pas^2 k + 4\Lambda (\pas^2 k)^2 -4 \Lambda \pas k \pas^3 k  
\\& +8\mu \pas ^2 k \pas^3 k + 2 \mu \Lambda (\pas k)^2 -2 \mu^2 \pas k \pas^2 k\, ,
\end{align}
that is,
\be\label{j=2Navier2}
-4 \pas^2 k \pas^5k =  12 (\pas k)^3 \pas^2 k - 2\Lambda (\pas^2 k)^2 +2\Lambda \pas k \pas^3 k 
 -4\mu \pas ^2 k \pas^3 k - \mu \Lambda (\pas k)^2 +\mu^2 \pas k \pas^2 k\, .
\ee
Hence, replacing the terms~\eqref{j=2Navier1} and~\eqref{j=2Navier2} in~\eqref{intj=2Navier}, we obtain
\begin{align}
   \dfrac{d}{dt} \int_\gamma\vert\pas^2 k\vert^2\de  s=&
\int_\gamma -4(\pas^4 k)^2 - 2 \mu(\pas^3k)^2 + \pol_{10}^3 (k) + \mu \pol_8 ^2 (k) \de s
 \\& -\Lambda (\pas^2 k )^2 \Big \vert_0 ^1 + 4 \Lambda \pas k \pas^3 k \Big \vert_0 ^1  -  \mu \Lambda(\pas k)^2 \Big \vert_0 ^1 + \mu^2 \pas k \pas^2 k\Big \vert_0 ^1\, .\label{intj=2final}
\end{align}
We now recall that $\Lambda$ is proportional to $\pas^2k$ at boundary points (see Lemma~\ref{tangent}), hence it follows that $\Lambda\pol_\sigma^h(k)=\pol_{\sigma+3}^{\max\{2, h\}}(k)$. Thus, we have
\begin{align}
   \dfrac{d}{dt} \int_\gamma\vert\pas^2 k\vert^2\de  s = &
-4\Vert \pas^4 k\Vert_{L^2(\gamma)}^2 - 2 \mu \Vert \pas^3k\Vert^2_{L^2(\gamma)} +  \int_\gamma\pol_{10}^3 (k) + \mu \pol_8 ^2 (k) \de s
 \\& + \pol_9^3(k) \Big \vert_0 ^1 + \mu \pol_7^3(k) \Big \vert_0 ^1 + \mu^2 \pol_5^2(k)\Big \vert_0 ^1\, .
\end{align}
By means of Lemma~4.6 and Lemma~4.7 in~\cite{ManPluPozSurvey}, for any $\varepsilon>0$ we have
\begin{align}
\int_{\gamma}^{} |\mathfrak{p}_{10}^{3}\left(k\right)|\de  s
\leq & \varepsilon \lVert\partial_s^{4}k\rVert_{L^2}^2+C(\varepsilon,\ell(\gamma))\left(\lVert k\rVert_{L^2}^2+\lVert k\rVert^{\Theta_1}_{L^2}\right)\,,\nonumber\\
\int_{\gamma}^{}|\mathfrak{p}_{8}^2\left(k\right)|\de  s
\leq& \varepsilon \lVert\partial_s^{3}k\rVert_{L^2}^2
+C(\varepsilon,\ell(\gamma))\left(\lVert k\rVert_{L^2}^2+C\lVert k\rVert^{\Theta_2}_{L^2}\right)\,,\nonumber
 \\\pol_{9}^{3}(k)(y)\vert
\leq &\varepsilon \lVert\pas^{4}k\rVert_{L^2}^2+C(\varepsilon,\ell(\gamma))\left(\lVert k\rVert_{L^2}^2+\lVert k\rVert^{\Theta_3}_{L^2}\right)\,,\nonumber\\
\pol_{7}^{3}(k)(y)
\leq &\varepsilon \lVert\pas^{4}k\rVert_{L^2}^2
+C(\varepsilon,\ell(\gamma))\left(\lVert k\rVert_{L^2}^2+C\lVert k\rVert^{\Theta_4}_{L^2}\right)\,,
\nonumber\\
\pol_{5}^{2}(k)(y)
\leq& \varepsilon \lVert\pas^{3}k\rVert_{L^2}^2
+C(\varepsilon,\ell(\gamma))\left(\lVert k\rVert_{L^2}^2+C\lVert k\rVert^{\Theta_5}_{L^2}\right)\, , 
\end{align}
for some exponents $\Theta_i >2$ with $i = 1 , \dots,5.$
\\Hence, we get
\be
   \dfrac{d}{dt} \int_\gamma\vert\pas^2 k\vert^2\de  s \leq 
-C\left ( \Vert \pas^4 k\Vert_{L^2(\gamma)}^2 + \mu \Vert \pas^3k\Vert^2_{L^2(\gamma)}\right) +C \left(\Vert k^2 \Vert^2_{L^2(\gamma)}+ \Vert k^2 \Vert^\Theta_{L^2(\gamma)} \right)
\ee
for some exponent $\Theta>2$ and constant $C$ which depend on $\ell(\gamma)$. Using the energy monotonicity proved in Proposition~\ref{energydecreases}, we conclude that
\be
   \dfrac{d}{dt} \int_\gamma\vert\pas^2 k\vert^2\de  s \leq C(\E(\gamma_0))\, .
\ee
\end{proof}

\subsection{Bound on $\Vert{\pas^6 k}\Vert _{L^2}$}
We observe that since~\eqref{motionequationtang} is a parabolic fourth-order equation, after having controlled the second-order derivative of the curvature, it is natural to control the sixth-order derivative of the curvature. Then, using interpolation inequalities, we get estimates for all the intermediate orders.
Before doing that, we notice that the elastic flow of curves becomes instantaneously smooth. More precisely, following the proof presented in~\cite{ManPluPozSurvey} in the case of closed curves (both using the so-called Angenent’s parameter trick~\cite{angenent,angen3,daprato-grisvard} and the classical theory of linear parabolic equations~\cite{solonnikov1}), one can show that given a solution to the elastic flow in a time interval $[0,T]$, then it is smooth for positive times, in the sense that it admits a $C^\infty$-parametrization in the interval $[\varepsilon,T]$ for every $\varepsilon \in (0,T)$. 

\medskip 

From now on, we denote by 
\be\label{velocity}
v:=\pat \gamma= V \nu + \Lambda \tau
\ee
the velocity of $\gamma$.
Hence, by means of integration by parts and the commutation rule in Lemma~\ref{evoluzionigeom}, we get the following identity
\begin{align}
\dfrac{d}{dt} \dfrac{1}{2} \int_\gamma \vert \pat^\perp v \vert ^2 \de  s =&- 2\int_\gamma \vert (\pas^\perp )^2 (\pat^\perp v) \vert ^2 \de  s 
\dfrac{1}{2}  \int_\gamma \vert\pat^\perp v \vert ^2 (\pas \Lambda-kV)  \de  s+ \int_\gamma \langle Y  , \pat^\perp v \rangle \de  s  
\\&
 - 2\langle  \pat^\perp v ,(\pas^\perp )^3 (\pat^\perp v)\rangle \Big \vert_0^1 + 2\langle \pas^\perp  (\pat^\perp v) , (\pas^\perp )^2 (\pat^\perp v) \rangle \Big \vert_0^1 \, , \label{eqcruciallemmaforv}
\end{align}
where we denoted by
\be\label{defY}
Y:= \pat ^\perp (\pat ^\perp v) + 2 (\pas ^\perp )^4 (\pat ^ \perp v)\, .
\ee
\\Before proceeding, we prove the following lemma, which gives estimates for some special family of polynomials.
\begin{lemma}\label{stimeintegraliebordo}
Let $\gamma:[0,1]\to\R^2$ be a smooth regular curve.
For all $j \leq 7 $, if the polynomial $\pol_{\sigma(j)}^j (k)$ defined as in Definition~\ref{polinomi}
satisfies one of the following conditions:
\begin{enumerate}[label=(\roman*)]
    \item\label{i} $\sigma(j) \geq 2(l+1)$ for all $l \leq j$, 
    \item\label{ii} $\sigma(j) \geq 2(l+1)$ for all $l \leq j-1$ and $(j+1) \leq \sigma(j) <  2(j+1)$,
\end{enumerate}
and
\be\label{iii}
\sigma(j) - \sum_{l=0}^j \alpha_l < 15\, ,
\ee 
then, there exists a constant $C$ and an exponent $\Theta>2$ such that
\be
\int_{\gamma}^{} |\mathfrak{p}_{\sigma(j)}^{j}\left(k\right)|\de  s
\leq \varepsilon \lVert\partial_s^{8}k\rVert_{L^2}^2+C(j,\varepsilon,\ell(\gamma))\left(\lVert k\rVert_{L^2}^2+\lVert k\rVert^{\Theta}_{L^2}\right)\,.\label{stimabis}
\ee
Similarly, for all $j \leq 7 $ and 
\be 
 \sigma'(j) - \sum_{l=0}^j \alpha_l < 16\, ,
\ee
there exists a constant $C$ and an exponent $\Theta'>2$ such that  for $y \in \{0,1\}$ it holds
\be
 \vert \pol_{\sigma'(j)}^{j}(k)(y)\vert
\leq \varepsilon \lVert\pas^{8}k\rVert_{L^2}^2+C(j,\varepsilon,\ell(\gamma))\left(\lVert k\rVert_{L^2}^2+\lVert k\rVert^{\Theta'}_{L^2}\right)\, . \label{stimabtbis}
\ee
\end{lemma}
\begin{proof}
By definition, every monomial of $\pol_{\sigma(j)}^{j}(k)$ is of the form
$C\prod_{l=0}^{j} (\pas^lk)^{\alpha_l}$ 
with $$\alpha_l\in\mathbb{N} \qquad \text{and} \qquad \sum_{l=0}^{j}\alpha_l(l+1)=\sigma(j)\, .$$
We set
$$
\beta_l:=\frac{\sigma(j)}{(l+1)\alpha_l}
$$
for every $l \leq j$ and we take $\beta_l=0$ if $\alpha_l=0$.
We observe that $\sum_{l\in J}^{}\frac{1}{\beta_l}=1$, hence by H\"older inequality, we get
 \be
C\int_{\gamma}\prod_{l=0}^{j} \vert \pas^lk\vert ^{\alpha_l}\,\mathrm{d}s
\leq C\prod_{l=0}^{j}\left(\int_{\gamma}\vert\pas^lk\vert^{\alpha_l\beta_l}\,\mathrm{d}s\right)^{\frac{1}{\beta_l}}
=C\prod_{l=0}^{j}\Vert\pas^lk\Vert^{\alpha_l}_{L^{\alpha_l\beta_l}}\,.
\ee
If condition~\ref{i} holds, then $\alpha_l\beta_l \geq 2$ for every $l\in J$. Applying the Gagliardo-Nirenberg inequality (see~\cite{AdamsFournier} or~\cite{aubin0}, for instance) for every $l \leq j$ yields
\be
\Vert\pas^lk\Vert_{L^{\alpha_l\beta_l}}\leq C(l,j,\alpha_l,\beta_l, \ell(\gamma))\Vert\pas^{8}k\Vert_{L^2}^{\eta_l}\Vert k\Vert_{L^2}^{1-\eta_l}+\Vert k \Vert_{L^2}\,
\ee
where the coefficient $\eta_l$ is given by
\be
\eta_l=\frac{l+1/2-1/(\alpha_l\beta_l)}{8} \in \left[ \frac{l}{8},1\right)\,.
\label{etal}
\ee
Then, we have 
\begin{align}
C\int_{\gamma}^{}\prod_{l=0}^{j}\vert\pas^lk\vert^{\alpha_l}\,\mathrm{d}s
&\leq C\prod_{l=0 }^{j}\Vert\pas^lk\Vert^{\alpha_l}_{L^{\alpha_l\beta_l}}\\
&\leq C\prod_{l=0}^{j}\Vert k\Vert^{(1-\eta_l)\alpha_l}_{L^2}\left(\Vert\partial_s^{8}k\Vert_{L^2}+\Vert k\Vert_{L^2}\right)^{\eta_l\alpha_l}_{L^2}\\
&= C\Vert k\Vert^{\sum_{l=0}^{j}(1-\eta_l)\alpha_l}_{L^2}\left(\Vert\pas^{8}k\Vert_{L^2}+\Vert k\Vert_{L^2}\right)^{\sum_{l=0}^{j}\eta_l\alpha_l}_{L^2}.
\end{align}
Moreover, from condition~\eqref{iii}, we have
\be
\sum_{l=0}^j \eta_l\alpha_l \leq \frac{\sigma(j)-1- \sum_{l=0}^j \alpha_l}{8} < 2\,,
\ee
that is, by means of Young's inequality with $p=\frac{2}{\sum_{l=0}^j\eta_l\alpha_l}$ and $q=\frac{2}{2-\sum_{l=0}^{j}\eta_l\alpha_l}$ we obtain
\be
C\int_{\gamma}\prod_{l=0}^{j}\vert\pas^lk\vert^{\alpha_l}\,\mathrm{d}s
\leq 
\varepsilon C \left(\Vert\pas^{8}k\Vert_{L^2}+\Vert k\Vert_{L^2}\right)^{2}_{L^2}+\frac{C}{\varepsilon}\Vert k\Vert^{\Theta}_{L^2}
\label{pintfinal}\ee
where constant $C$ depends on $j$, $\varepsilon$, $\ell (\gamma)$ and $\Theta>2$.

Otherwise, if condition~\ref{ii} holds, we have $1\leq \alpha_j\beta_j<2$, that is
\be
\Vert \pas ^j k \Vert_{L^{\alpha_j \beta_j}}^{\alpha_j} \leq \Vert \pas ^j k \Vert_{L^2}^{\alpha_j} \leq \Vert\pas^{8}k\Vert^{\eta_j\alpha_j}_{L^2} \Vert k\Vert^{(1-\eta_j)\alpha_j}_{L^2}+\Vert k\Vert^{\alpha_j}_{L^2} 
\ee
where $\eta_j= \frac{j}{8}$ and we used the boundedness of $\ell(\gamma)$.

Hence, as in the previous case, we have
\begin{align}
C\int_{\gamma}\prod_{l=0}^{j}\vert\pas^lk\vert^{\alpha_l}\,\mathrm{d}s
\leq& C \prod_{l=0}^{j-1} \Vert \pas^l k \Vert_{L^{\alpha_l \beta_l}}^{\alpha_l} \Vert \pas^j k \Vert_{L^{\alpha_j \beta_j}}^{\alpha_j}
\\ \leq & C \left (\Vert\pas^{8}k\Vert^{\sum_{l=0}^{j-1}\eta_l\alpha_l}_{L^2} \Vert k\Vert^{\sum_{l=0}^{j-1}(1-\eta_l)\alpha_l}_{L^2}+\Vert k\Vert_{L^2}^{\sum_{l=0}^{j-1}{\alpha_l}} \right) 
\\ & \phantom{C}\left (\Vert\pas^{8}k\Vert^{\eta_j\alpha_j}_{L^2} \Vert k\Vert^{(1-\eta_j)\alpha_j}_{L^2}+\Vert k\Vert^{\alpha_j}_{L^2} \right) 
\\ \leq& C \Vert k\Vert_{L^2}^{\sum_{l=0}^j(1-\eta_l)\alpha_l } \left  (\Vert\pas^{8}k\Vert^{\sum_{l=0}^{j}\eta_l\alpha_l}_{L^2}+\Vert\pas^{8}k\Vert^{\sum_{l=0}^{j-1}\eta_l\alpha_l}_{L^2}
\Vert k\Vert^{\eta_j\alpha_j}_{L^2}
\right.
\\ & \left.\phantom{ C \Vert k\Vert_{L^2}^{\sum_{l=0}^j(1-\eta_l)\alpha_l } (}  + \Vert\pas^{8}k\Vert^{\eta_j\alpha_j}_{L^2}  \Vert k\Vert_{L^2}^{\sum_{l=0}^{j-1}{\eta_l \alpha_l}}+ \Vert k\Vert_{L^2}^{\sum_{l=0}^{j}{\eta_l \alpha_l}} \right)
\end{align}
where for all $l\leq j-1$ the coefficient $\eta_l$ is given by expression~\eqref{etal} and $\eta_j= \frac{j}{8}$.
Applying again Young's inequality, since $\sum_{l=0}^{j}\eta_l\alpha_l<2$ still holds, we obtain estimate~\eqref{pintfinal}. 

The second part of the lemma comes using the same arguments.
\end{proof}
    \begin{lemma}\label{Tpol}
 Let $\gamma_t$ be a maximal solution to the elastic flow of curves subjected to Navier boundary conditions~\eqref{navierbc}, such that the {\em uniform non-degeneracy condition}~\eqref{defunifnondeg} holds in the maximal time interval $[0,T_{\max})$. Then, for every $j \in \N$, it holds
\be\label{Tpoleq}
\pas^j \Lambda (y)\pol_\sigma^h(k)(y) = \pol_{\sigma+ j+3}^{\max\{h, j+2\}}(k)(y)
\ee
and
\be\label{patTpoleq}
\pat \Lambda (y)\pol_\sigma^h(k)(y) = \pol^{\max \{h,6\}}_{\sigma+7}(k)(y)+ \mu \pol^{\max \{h,4\}}_{\sigma+5}(k)(y)
\ee
for every $t \in[0,T_{\max})$ and $y \in \{0,1\}$.
\end{lemma}
\begin{proof}
 By means of Lemma~\ref{tangent} and by the fact that $\tau_2$ is bounded from below, we have
    \be 
    \Lambda(y) =  \pas^2 k (y) \dfrac{\nu_2(y)}{\tau_2(y)} = \pol^2_3(k)(y)
    \ee
    for $y \in \{0,1\}$.
    Hence, by Remark~\ref{prop-polinomi}, it follows 
    \be
    \pas^j \Lambda(y) = \pol^{j+2}_{j+3}(k)(y)\, ,
    \ee
and thus,
     \be
    \pas^j \Lambda(y) \pol_\sigma^h(k)(y)=  \pol^{j+2}_{j+3}(k)(y)\pol_\sigma^h(k)(y)= \pol^{\max \{h,j+2\}}_{\sigma+j+3}(k)(y)\, .
    \ee
    Similarly, by formula~\eqref{dtpolinomi}, we have
    \be \pat \Lambda (y)= \pat \Big(\pol_3^2(k)(y)\Big)= \pol_7^6 (k)(y)+  \mu \pol_5^4 (k)(y)
    \ee
   then,
     \be
    \pat \Lambda(y) \pol_\sigma^h(k)(y)= \pol^{\max \{h,6\}}_{\sigma+7}(k)(y)+ \mu \pol^{\max \{h,4\}}_{\sigma+5}(k)(y)\, .
    \ee
\end{proof}
From now on, for any $t \in [0,T_{\max})$, we choose the tangential velocity $\Lambda(t,x)$ with $x \in (0,1)$ as the linear interpolation between the value at the boundary points, that is
\be\label{Lambda}
\Lambda(t,x)= \Lambda(t,0) \Big( 1 + \frac{\Lambda(t,1)-\Lambda(t,0)}{\Lambda(t,0)} \frac{1}{\ell(\gamma)} \int_0^x \vert \pa_x \gamma \vert \, dx \Big )\,.
\ee 
\begin{lemma}\label{stimelambda}
    Let $\Lambda$ be the tangential velocity defined in~\eqref{Lambda}, there exist two constants $C_1=C_1(\ell(\gamma))$ and $C_2=C_2(\E(\gamma_0),\ell(\gamma))$ such that
    \begin{align}
      \vert \pas \Lambda (t,x)\vert &\leq C_1(\vert \Lambda(t,1) \vert + \vert \Lambda(t,0) \vert) \, , \label{stimepasTinside}
    \\ \vert \pat \Lambda (t,x)\vert &  \leq C_2 \Big [ \vert \pat \Lambda (t,0)\vert +\vert \pat \Lambda (t,1)\vert + \vert \pat \Lambda (t,0)\vert \frac{\vert  \Lambda (t,1)\vert } {\vert \Lambda (t,0)\vert }
    \\&\phantom{\leq}+ \vert \Lambda(t,1)-\Lambda(t,0)\vert ^2+ \vert \Lambda(t,1)-\Lambda(t,0)\vert
    \Big] 
    \label{stimepatTinside}
\end{align}
for $t \in [0,T_{\max})$ and $x \in [0,1]$.
\end{lemma}
\begin{proof}
From~\eqref{Lambda} it easily follows that
\be\label{pasLambda}
\pas \Lambda (t,x) = \frac{\Lambda(t,1)-\Lambda(t,0)}{\ell(\gamma)} \qquad \text{and} \qquad \pas^j \Lambda(t,x) =0 \quad \text{for $j \geq 2$.}
\ee
Moreover, taking the time derivative, we get
\begin{align}
    \pat \Lambda (t,x)=& \pat \Lambda(t,0) \Big( 1 + \frac{\Lambda(t,1)-\Lambda(t,0)}{\Lambda(t,0)} \frac{1}{\ell(\gamma)} \int_0^x \vert \pa_x \gamma \vert \, dx \Big ) 
    \\&+ \frac{(\pat \Lambda(t,1)-\pat \Lambda(t,0)) \Lambda(t,0)-(\Lambda(t,1)-\Lambda(t,0)) \pat \Lambda(t,0) }{\Lambda(t,0)} \frac{1}{\ell(\gamma)} \int_0^x \vert \pa_x \gamma \vert \, dx 
    \\&-  \big(  \Lambda(t,1)-\Lambda(t,0) \big)\frac{1}{\ell^2(\gamma)} \frac{d (\ell(\gamma))}{dt}\int_0^x \vert \pa_x \gamma \vert \, dx
    \\&+ \big( \Lambda(t,1)-\Lambda(t,0)\big)\frac{1}{\ell(\gamma)} \frac{d }{dt}\int_0^x \vert \pa_x \gamma \vert \, dx 
    \\=&\pat \Lambda(t,0) \Big( 1 + \frac{\Lambda(t,1)-\Lambda(t,0)}{\Lambda(t,0)} \frac{1}{\ell(\gamma)} \int_0^x \vert \pa_x \gamma \vert \, dx \Big ) 
    \\&+ \frac{(\pat \Lambda(t,1)-\pat \Lambda(t,0)) \Lambda(t,0)-(\Lambda(t,1)-\Lambda(t,0)) \pat \Lambda(t,0) }{\Lambda(t,0)} \frac{1}{\ell(\gamma)} \int_0^x \vert \pa_x \gamma \vert \, dx 
    \\&-  \frac{\big(  \Lambda(t,1)-\Lambda(t,0) \big)^2}{\ell^2(\gamma)}\int_0^x \vert \pa_x \gamma \vert \, dx +  \frac{  \Lambda(t,1)-\Lambda(t,0) } {\ell^2(\gamma)} \int_\gamma k V \de s \int_0^x \vert \pa_x \gamma \vert \, dx
\\&+ \big( \Lambda(t,1)-\Lambda(t,0)\big)\frac{1}{\ell(\gamma)} \frac{d }{dt}\int_0^x \vert \pa_x \gamma \vert \, dx 
      \, .\label{patLambda}
\end{align}
where we used relations~\eqref{metricaevol}.
Hence, noticing that from interpolation and Proposition~\ref{propNavierj=2} it follows
\be
\int_\gamma kV \leq C(\E(\gamma_0))\, ,
\ee
we obtain the last estimate in the statement.
\end{proof}
\begin{lemma} If $\gamma$ satisfies~\eqref{motionequationtang}, then 
    \begin{align}
\pat^2 \gamma = &\Big( 4 \pas^6 k +10 k^2 \pas^4 k + \pol_7^3(k)  -4 \Lambda \pas^3 k -6  \Lambda k^2 \pas k + \Lambda^2 k
   \\&- 4 \mu \pas^4 k + \mu \pol_5^2(k)+ 2 \mu \Lambda \pas k + \mu^2 \pas^2 k + \mu^2 \pol_3^0(k) \Big) \nu 
   \\&+ \Big (\pat \Lambda + \pol_7^3(k)-2 \Lambda k \pas^3 k -3 \Lambda k^3 \pas k + \mu \pol_5^3(k)+\mu \Lambda k \pas k+ \mu^2 \pol_3^1(k)\Big) \tau\, .\label{dt2gammaesteso}
\end{align}
\end{lemma}
\begin{proof}
We firstly compute
 \be\label{patV}
   \pat V = 4 \pas^6 k + \pol_7^3(k)-2 \Lambda \pas^3 k + \Lambda \pol_4^1(k) - 4 \mu \pas^4 k + \mu \pol_5^2(k) + \mu \Lambda \pas k + \mu^2 \pas^2 k + \mu^2 \pol_3^0(k) 
 \ee
 and
\begin{align}\label{paspatV}
    \pat \pas V =& 4 \pas^7 k + \pol_8^4(k) -2 \Lambda \pas^4 k + \Lambda \pol_5^2(k)+ \pas \Lambda \pol_4^1(k)-4 \mu \pas^5 k + \mu \pol_6^3(k)
    \\&+ \mu \Lambda \pas^2 k +3 \mu \pas \Lambda \pas k + \mu^2 \pas^3 k + \mu^2 \pol_4^1(k)\, .
\end{align}
Then, by means of Lemma~\ref{evoluzionigeom}, we have
\begin{align}
    \pat^2 \tau=& ( \pat\pas V + \pat \Lambda k + \Lambda \pat k)\nu - (\pas V + \Lambda k)^2 \tau
    \\=&\Big (4 \pas^7 k + \pol_8^4(k) -4 \Lambda \pas^4 k + \Lambda \pol_5^2(k)+ \pas \Lambda \pol_4^1(k)+\Lambda^2\pas k+\pat \Lambda k-4 \mu \pas^5 k + \mu \pol_6^3(k)
    \\&\phantom{\Big(}+ \mu \Lambda \pol_3^2 (k) +3 \mu \pas \Lambda \pas k + \mu^2 \pas^3 k + \mu^2 \pol_4^1(k)\Big) \nu
    \\ &+ \Big ( \pol_8^3(k) + \Lambda \pol_5^3(k)+ \Lambda^2 k^2 + \mu \pol_6^3(k)+ \mu \Lambda \pol_3^2(k) + \mu^2 \pol_4^1(k)\Big) \tau \, .\label{pat2tau}
\end{align}
Similarly, differentiating in time the relation~\eqref{patnu} we get
\be
\pat^2 \nu = -(\pas V + \Lambda k)^2 \nu - (\pat \pas V + \pat \Lambda k + \Lambda \pat k) \tau\, ,
\ee
that is
\begin{align}
    \pat^2 \nu=&\Big ( \pol_8^3(k) + \Lambda \pol_5^3(k)+ \Lambda^2 k^2 + \mu \pol_6^3(k)+\mu \Lambda \pol_3^2(k) + \mu^2 \pol_4^1(k)\Big) \nu 
    \\&- \Big (4 \pas^7 k + \pol_8^4(k) -4 \Lambda \pas^4 k + \Lambda \pol_5^2(k)+ \pas \Lambda \pol_4^1(k)+\Lambda^2\pas k+\pat \Lambda k-4 \mu \pas^5 k + \mu \pol_6^3(k)
    \\&\phantom{-\Big(}+ \mu \Lambda \pol_3^2 (k) +3 \mu \pas \Lambda \pas k + \mu^2 \pas^3 k + \mu^2 \pol_4^1(k)\Big) \tau \, .\label{pat2nu}
\end{align}

Using computations~\eqref{pat2tau} and~\eqref{pat2nu}, we obtain 
\begin{align}
\pat^2 \gamma =& (\pat V + \Lambda (\pas V + \Lambda k ) ) \nu + (\pat \Lambda - V ( \pas V -  \Lambda k)) \tau 
\\=&  \Big(\pat V+ \Lambda (-2\pas^3 k -3 k^2 \pas k + \mu \pas k)+ \Lambda^2 k\Big) \nu 
\\&+ \Big(\pat \Lambda - (-2 \pas^2 k - k^3 - \mu k)( -2 \pas^3 k -3k^3 \pas k - \mu \pas k - \Lambda k)\Big) \tau \\ {=} &\Big( 4 \pas^6 k +10 k^2 \pas^4 k + \pol_7^3(k)  -4 \Lambda \pas^3 k -6  \Lambda k^2 \pas k + \Lambda^2 k
   \\&- 4 \mu \pas^4 k + \mu \pol_5^2(k)+ 2 \mu \Lambda \pas k + \mu^2 \pas^2 k + \mu^2 \pol_3^0(k) \Big) \nu 
   \\&+ \Big (\pat \Lambda + \pol_7^3(k)-2 \Lambda k \pas^3 k -3 \Lambda k^3 \pas k + \mu \pol_5^3(k)+\mu \Lambda k \pas k+ \mu^2 \pol_3^1(k)\Big) \tau\, .
\end{align}
\end{proof}
In the following, we show that to estimate the $L^2$-norm of $\pas^6 k$ it is enough to control the $L^2$-norm of $\pat ^\perp v$. Hence, we start writing the boundary terms in~\eqref{eqcruciallemmaforv} using the curvature and its derivatives, lowering the order by means of the boundary condition.
\begin{lemma}\label{lemmabt1}
 Let $\gamma_t$ be a family of curves moving with velocity $v$ defined in~\eqref{velocity}. Then,
    \begin{equation}
\langle \pas^\perp  (\pat^\perp v) , (\pas^\perp )^2 (\pat^\perp v) \rangle = \pol_{17}^7 (k)+ \pol_{15}^7 (k) +\pol_{13}^7 (k)+ \pol_{11}^5 (k) + \mu^4 \pol_{9}^4 (k) \, .\label{btinferiorej=6}
\end{equation}
\end{lemma}
\begin{proof}
    By straightforward computations, we have that
\be
\pas^\perp (\pat^\perp v) =\pas(\pat v^\perp) \nu \, , \quad (\pas^\perp )^2 (\pat^\perp v) = \pas^2 (\pat v^\perp) \nu 
\ee
where $\pat v^\perp$ is the normal component of $\pat^2 \gamma$, which is computed in~\eqref{dt2gammaesteso}.
Hence, we compute
\begin{align}
\pas(\pat v^\perp)=&
4 \pas^7 k  + \pol_8^4(k)-4 \Lambda \pas^4 k - 4 \pas \Lambda \pas^3 k + \Lambda^2 \pas k -4 \mu \pas^5 k + \mu \pol_7^4(k)
\\& + 2 \mu \Lambda \pas^2 k +2 \mu \pas \Lambda \pas k+ \mu^2 \pas^3 k+ \mu^2\pol_4^1(k) 
  \label{dsperpdtperpv}
\end{align}
and
\begin{align}
\pas^2 (\pat^\perp v) =& \pol_9^5(k) + 4\Lambda \pas \Lambda \pas k
+ \Lambda \pol_6^2(k) -4 \pas^2 \Lambda \pas^3 k - 8 \pas \Lambda \pas^4 k   - \pat \Lambda \pas k \\& +\mu \pol_7^4(k)  + \mu \Lambda \pol_4^1(k)
+2 \mu \pas^2 \Lambda \pas k + 4 \mu \pas \Lambda \pas^2 k + \mu^2 \pol_5^2(k)\, ,\label{ds2perpdtperpv}
\end{align}
where in relation~\eqref{ds2perpdtperpv} we used
\begin{align}
    4 \pas^8 k =& \pol_9^5(k) + 4 \Lambda \pas^5 k + \Lambda \pol_6^2(k) - \Lambda^2 \pas^2 k - \pat \Lambda \pas k \\&+ 4 \mu \pas^6 k +\mu \pol_7^4(k) - 2 \mu \Lambda \pas^3 k + \mu \Lambda \pol_4^1(k) - \mu^2 \pas^4 k + \mu^2 \pol_5^2(k)
\end{align}
 since Navier boundary conditions hold.
 So, using expressions~\eqref{dsperpdtperpv} and~\eqref{ds2perpdtperpv}, replacing $\Lambda$ and its derivatives by means of Lemma~\ref{Tpol} and recalling that $\mu>0$ is constant, we get
\begin{equation}
\langle \pas^\perp  (\pat^\perp v) , (\pas^\perp )^2 (\pat^\perp v) \rangle = \pol_{17}^7 (k)+ \pol_{15}^7 (k) +\pol_{13}^7 (k)+ \pol_{11}^5 (k) + \mu^4 \pol_{9}^4 (k) \, .
\end{equation}
\end{proof}

\begin{lemma}\label{lemmabt2}
 Let $\gamma_t$ be a family of curves moving with velocity $v$ defined in~\eqref{velocity}. Then,
    \begin{align}
\langle  \pat^\perp v ,(\pas^\perp )^3 (\pat^\perp v)\rangle =& \langle  \pat ^\perp v , 4\pas^9 k \nu\rangle
+ \langle  \pat^\perp v , (\pas^\perp )^3 (\pat^\perp v)-4\pas^9 k \nu\ \rangle 
\\=&\pol_{17}^6 (k)+  \pol_{15}^7(k)   + \pol_{13}^7(k) +\pol_{11}^7(k) + \pol_{9}^5(k) + \pol_7^2(k)\, . \label{btsuperiorej=6}\end{align} 
\end{lemma}
\begin{proof}
    Let us analogously handle the other boundary term in~\eqref{eqcruciallemmaforv}. By standard computations, we have that 
\be
(\pas^3)^\perp (\pat^\perp v) = \pas^3 (\pat v^\perp ) \nu
\ee
where
\begin{align}
\pas^3 (\pat v^\perp )=& 4 \pas^9 k+ \pol_{10}^6(k) -4 \Lambda \pas^6 k+ \pas \Lambda \pol_6^5(k)-12 \pas^2 \Lambda \pas^4 k - 4 \pas^3 \Lambda \pas^3 k 
\\&+ \Lambda^2 \pas^3 k+ 6 \Lambda \pas \Lambda \pas^2 k+ 6\Lambda \pas^2 \Lambda \pas k 
- 4 \mu \pas^7 k + \mu \pol_8^5(k) \\& +\mu  2 \Lambda \pas^4 k+ 6 \mu  \pas \Lambda \pas^3 k+ 6 \mu \pas^2 \Lambda \pas^2 k+  2 \mu\pas^3 \Lambda \pas k   + \mu^2 \pas^5 k+ \mu^2 \pol_6^3(k)\, .\label{ds3perpdtperpv}
\end{align}
As above, we aim to write the ninth-order derivative as the sum of lower-order derivatives.
\\Hence, from condition~\eqref{navierbc}, at boundary points it holds
\be\label{condnavierj=0}
\langle  \pat v , \pat^2 (-2 \pas k \nu + \mu \tau  )\rangle =0 \,,
\ee 
where
\begin{align}\label{patv}
    \pat v =& \pat V \nu + V \pat \nu + \pat \Lambda \tau + \Lambda \pat \tau
    \\=& \Big(\pat V + \Lambda \pas V \Big) \nu + (\pat \Lambda -V \pas V) \tau
    \\=& \Big(  4 \pas^6 k + \pol_7^3(k)-4 \Lambda \pas^3 k + \Lambda \pol_4^1(k) - 4 \mu \pas^4 k + \mu \pol_5^2(k) + 2\mu \Lambda \pas k + \mu^2 \pas^2 k + \mu^2 \pol_3^0(k)\Big )\nu
    \\&+ \Big (\pat \Lambda -4 \pas^2k \pas^3 k \Big) \tau
\end{align}
and
\begin{align}
\pat^2 (-2 \pas k \nu + \mu \tau)=& -2 \pat^2 \pas k \nu -4 \pat \pas k \pat\nu  -2 \pas k \pat^2 \nu + \mu \pat^2 \tau 
\\=& \Big (-2 \pat^2 \pas k-2 \pas k (\pat^2 \nu) ^\perp+ \mu (\pat^2 \tau)^\perp \Big) \nu
\\&+ \Big(4 \pas V\pat \pas k -2 \pas k (\pat^2 \nu)^\top+ \mu (\pat^2 \tau)^\top \Big) \tau \, .
\end{align}
Then, after computing $\pat^2 \pas k$, we have
\begin{align}
\pat^2 (-2 \pas k \nu + \mu \tau)=&\Big (-8 \pas^9 k+ \pol_{10}^6(k)   +\Lambda \pol_7^6(k)+ \pas \Lambda \pol_6^2(k)+ \Lambda^2 \pol_4^3(k) +(\pas \Lambda)^2 \pol_2^1(k)
\\&\phantom{\Big(}+\pat \Lambda \pol^2_3(k)  
 + \mu \pol_8^7(k) 
 + \mu \Lambda \pol_5^4(k) + \mu \pas \Lambda \pol_4^1(k) +\mu \Lambda^2\pol^1_2(k)
 \\&\phantom{\Big(}+\mu \pat \Lambda \pol_1^0(k)
 + \mu^2 \pol_6^5(k)+ \mu^2 \Lambda \pol_3^2 (k) + \mu^2 \pas \Lambda \pol_2^1(k) 
 + \mu^3 \pol_4^3(k) \Big) \nu
\\&+ \Big(\pol_{10}^7 (k) +\Lambda \pol_7^4(k)+ \pas \Lambda \pol_6^1(k)+ \Lambda^2 \pol_4^1(k)+ \pat \Lambda \pol_3^1(k)
\\&\phantom{\Big(}+\mu \pol_8^3(k) + \mu \Lambda \pol_5^3(k)+ \mu \pas \Lambda \pol_4^1(k) + \mu \Lambda^2 \pol_2^0(k)
\\&\phantom{\Big(}+\mu^2 \pol_6^3(k) + \mu^2 \Lambda\pol_3^2(k)+ \mu^3 \pol_4^1(k)
\Big) \tau \, .\label{pat2bc}
\end{align}
Replacing equations~\eqref{patv} and~\eqref{pat2bc} in the scalar product~\eqref{condnavierj=0} and recalling that at boundary points $\Lambda$ and its derivatives can be approximated by suitable polynomials (as it is shown in Lemma~\ref{Tpol}), we get 
\be
       \langle \pat v,8 \pas^9 k \nu \rangle=  \pol_{17}^6 (k)+  \pol_{15}^7(k) + \pol_{13}^7(k) + \pol_{11}^7(k)+ \pol_{9}^5(k)+ \pol_7^2(k)\, .
\ee
We now notice that
\be
\langle \pat v , \pas^9 k \nu \rangle = \langle \pat v^\perp \nu + \pat v^\top \tau, \pas^9 k \nu \rangle=\langle \pat^\perp v, \pas^9 k \nu\rangle    \, ,
\ee
hence, we have 
\begin{align}
\langle  \pat^\perp v ,(\pas^\perp )^3 (\pat^\perp v)\rangle =& \langle  \pat ^\perp v , 4\pas^9 k \nu\rangle
+ \langle  \pat^\perp v , (\pas^\perp )^3 (\pat^\perp v)-4\pas^9 k \nu\ \rangle 
\\=&\pol_{17}^6 (k)+  \pol_{15}^7(k)   + \pol_{13}^7(k) +\pol_{11}^7(k) + \pol_{9}^5(k) + \pol_7^2(k)\, . \end{align} 
\end{proof}
\begin{prop}\label{propNavierJ=6}
Let $\gamma_t$ be a maximal solution to the elastic flow of curves subjected to boundary conditions~\eqref{navierbc}, with initial datum
$\gamma_0$ in the maximal time interval $[0,T_{\max})$.
Then for all $t\in (0,T_{\max})$ it holds
\begin{equation}\label{stimaNavierj=6}
\int_{\gamma} \vert\pat ^\perp v \vert^2\de  s
\leq C ( \mathcal{E}(\gamma_0)) \, .
\end{equation}
\end{prop}
\begin{proof}
The thesis follows once we estimate the quantities in~\eqref{eqcruciallemmaforv}.
From equation~\eqref{ds2perpdtperpv}, we have
\begin{align}
- 2\int_\gamma \vert (\pas^\perp )^2 (\pat^\perp v) \vert ^2 \de  s = - 2\int_\gamma \vert &4 \pas^8 k + \pol_{9}^5(k)+ \Lambda \pol_{6}^5 (k) + \pas \Lambda \pol_{5}^4 (k)
\\&+ \pas^2 \Lambda \pol_{4}^3 (k)
 + \Lambda^2 \pol_{3}^2 (k)+ \Lambda \pas \Lambda \pol_{2}^1 (k)
\\&+ \mu \pol_{7}^6 (k) 
+ \mu \Lambda \pol_{4}^3 (k)
+\mu  \pas \Lambda \pol_{3}^2 (k)
+ \mu^2 \pol_{5}^4 (k) \vert ^2 \de  s\, .
\end{align}
Hence, using the simple inequalities
\begin{align}
\vert a+b \vert^2 &\leq C \big ( \vert a \vert ^2 + \vert b \vert ^2 \big)\, ,
\\\vert a+b \vert^2 & \geq (1-\varepsilon) \vert a \vert ^2 - C(\varepsilon) \vert b \vert ^2
\end{align}
with $\varepsilon=\dfrac{1}{2}$, we get
\begin{align}
- 2\int_\gamma \vert (\pas^\perp )^2 (\pat^\perp v) \vert ^2 \de  s  \leq&  - \int_\gamma \vert  4 \pas^8 k \vert ^2 + C  \int_\gamma   \vert \pol_{18}^5(k)+\Lambda^2 \pol_{12}^5 (k) + (\pas \Lambda )^2 \pol_{10}^4 (k)
\\ &\phantom{ - \int_\gamma \vert  4 \pas^8 k \vert ^2 + C  \int_\gamma }  + (\pas^2 \Lambda)^2 \pol_{8}^3 (k) + \Lambda^4 \pol_6^2(k)
+ \mu^2 \pol_{14}^6 (k) 
\\ &\phantom{ - \int_\gamma \vert  4 \pas^8 k \vert ^2 + C  \int_\gamma } +\mu^2  \Lambda^2 \pol_{8}^3 (k) + \mu^2 (\pas \Lambda )^2 \pol_6^2 (k) + \mu^4 \pol_{10}^4(k) \de  s 
\\ \leq&  - 16\int_\gamma \vert  \pas^8 k \vert ^2 +  \int_\gamma  \vert  \pol_{18}^5(k)+\pol_{16}^4 (k) + \pol_{14}^6 (k)+ \pol_{12}^2(k)+\pol_{10}^4 (k) \vert \de  s \,,  \label{integraldsperp2}
\end{align}
where we used the very expression of $\Lambda$ in~\eqref{Lambda} and the estimates in Lemma~\ref{stimelambda}.

Moreover, with same arguments, from equation~\eqref{patv} we get 
\begin{align}
\dfrac{1}{2} \int_\gamma \vert\pat^\perp v \vert ^2 (\pas \Lambda -  kV ) \de  s = & \dfrac{1}{2} \int_\gamma \vert\pat^\perp v \vert ^2 (\pas \Lambda +2 k \pas^2 k + k^4 -\mu k^2 ) \de  s
\\=&\int_\gamma\vert \pol_{18}^6 (k)+\pol_{17}^6 (k) + \pol_{16}^6 (k)+ \pol_{15}^6 (k)+\pol_{14}^6 (k)+\pol_{13}^6 (k)
\\&\phantom{\int_\gamma\vert}+\pol_{12}^6 (k)+ \pol_{12}^2 (k)+\pol_{10}^4 (k)+\pol_{9}^2 (k)+\pol_8^2(k)\vert \de  s\, .
    \label{integraldtperpv2}
    \end{align}
We only need to compute the integral in~\eqref{eqcruciallemmaforv} involving $Y$. By straightforward computation, we have 
\begin{align}
\pat (\pat v)^\perp =& -8 \pas^{10} k -20 k^2 \pas^8 k + \pol_{11}^7(k) + \Lambda \pol_8^7(k) + \Lambda^2 \pol_5^4(k) + \pat \Lambda \pol_4^3(k) 
\\&+ 12 \mu  \pas^8 k+\mu  \pol_9^6(k)+\mu \Lambda \pol_6^5(k)+\mu \Lambda^2 \pol_3^2(k) + \mu \pat \Lambda \pol_2^1(k) + \mu^2\pol_7^6(k)
\\&+ \mu^2 \Lambda \pol_4^3(k)+ \mu^3 \pol_5^4(k) \, , \label{dt2vperp}
\end{align}
and 
\begin{align}
    \pas ^4 (\pat v) ^ \perp=& 4 \pas^{10} k + \pol_{11}^7(k) + \Lambda \pol_8^7(k)+ \pas \Lambda \pol_7^6(k) + \pas^2 \Lambda \pol_6^5(k) + \pas^3 \Lambda \pol_5^4(k)+ \pas^4 \Lambda \pol_4^3(k) 
\\&-4 \mu  \pas^8 k+\mu  \pol_9^6(k)+\mu \Lambda \pol_6^5(k)+\mu \pas \Lambda \pol_5^4(k)+\mu  \pas^2 \Lambda \pol_4^3(k) 
\\&+\mu  \pas^3 \Lambda \pol_3^2(k) + \mu \pas^4 \Lambda \pol_2^1(k) + \mu^2(\pol_7^6(k)
     \, . \label{ds4dtvperp}
\end{align}
Then, 
we get
\begin{align}
Y=&\pat^\perp  (\pat ^\perp v)+2 (\pas ^\perp )^4 (\pat ^ \perp v) 
\\= &\Big( -20 k^2 \pas^8 k + \pol_{11}^7(k) + \Lambda \pol_8^7(k) + \pas \Lambda \pol_7^6(k) + \Lambda^2 \pol_5^4(k) + \pat \Lambda \pol_4^3(k) 
\\&\phantom{\Big(}
+ 4 \mu  \pas^8 k+\mu  \pol_9^6(k)+\mu \Lambda \pol_6^5(k)
+\mu \pas \Lambda \pol_5^4(k)+\mu \Lambda^2 \pol_3^2(k) + \mu \pat \Lambda \pol_2^1(k) 
\\&\phantom{\Big(}+ \mu^2\pol_7^6(k)+ \mu^2 \Lambda \pol_4^3(k)+ \mu^3 \pol_5^4(k)\Big)  \nu \, .\label{Y}
\end{align}

Hence, computing the scalar product $\langle Y , \pat ^\perp v \rangle$, using the well-known Peter-Paul inequality and integrating by parts the integral $\int_\gamma \pas^6 k \pas^8 k \de s$, we have 
\begin{align}
    \int_\gamma \langle Y  , \pat^\perp v \rangle \de  s \leq & \frac{1}{2} \int_\gamma \vert \pas^8 k \vert^2 \de s -4 \mu \int_\gamma \vert \pas^7 k \vert ^2 \de s
    + \pol_{15}^7(k) \Big \vert_0^1
 \\& +\int_\gamma \vert \pol_{18}^7(k)+ \pol_{17}^6(k) + \pol_{18}^7(k) \pol_{16}^7(k)+\pol_{15}^6(k)+\pol_{14}^7(k)
 \\&\phantom{\int_\gamma \vert}+\pol_{13}^6(k)+ \pol_{12}^6(k)+ \pol_{11}^4(k)+ \pol_{10}^4(k)+\pol_{8}^4(k)+ \pol_{6}^2(k) \vert \de  s \, .
    \label{intY}
\end{align}
where, as above, we estimated $\Lambda$ and its derivatives by means of Lemma~\ref{stimelambda}. 

Moreover, using identities in Lemma~\ref{lemmabt1} and Lemma~\ref{lemmabt2}, we end up with the following inequality
\begin{align}
     - 2\langle  \pat^\perp v ,(\pas^\perp )^3 (\pat^\perp v)\rangle \Big \vert_0^1 + 2\langle \pas^\perp  (\pat^\perp v) , (\pas^\perp )^2 (\pat^\perp v) \rangle \Big \vert_0^1 \leq& \vert \pol_{17}^7(k)\vert  +  \vert \pol_{15}^7(k)\vert +  \vert \pol_{13}^7(k)\vert 
     \\&+  \vert \pol_{11}^7(k)\vert + \vert \pol_9^5(k)\vert  + \vert  \pol_7^2(k) \vert \, . \label{btfinalej=6}
\end{align}
Then, putting together inequalities~\eqref{integraldsperp2},~\eqref{integraldtperpv2},~\eqref{intY} and~\eqref{btfinalej=6}, we get 
\begin{align}
\dfrac{d}{dt} \dfrac{1}{2} \int_\gamma \vert \pat^\perp v \vert ^2 \de  s \leq & \int_\gamma \vert \pol_{18}^7(k)+ \pol_{17}^7(k)+\pol_{16}^7(k)+\pol_{14}^7(k)+ \pol_{15}^6(k) + \pol_{13}^6(k) +  \pol_{12}^6(k)
\\& \phantom{\int_\gamma}+ \pol_{11}^4(k)+ \pol_{10}^4(k) +\pol_8^4(k) + \pol_{9}^2(k)+  \pol_{6}^2(k) \vert \de s 
\\& +\vert \pol_{17}^7(k)\vert  + \mu \vert \pol_{15}^7(k)\vert + \mu^2 \vert \pol_{13}^7(k)\vert + \mu^3 \vert \pol_{11}^7(k)\vert + \mu^4 \vert \pol_9^5(k)\vert  + \mu^5 \vert  \pol_7^2(k) \vert \Big \vert_0^1\, .
\end{align}
By means of Lemma~\ref{stimeintegraliebordo}, we have
\begin{align}
\dfrac{d}{dt} \dfrac{1}{2} \int_\gamma \vert \pat^\perp v \vert ^2 \de  s \leq &- C \Big(\Vert \pas^8 k \Vert^2 _{L^2(\gamma)}+\mu  \Vert \pas^7 k \Vert^2 _{L^2(\gamma)} \Big) + C\Big( \Vert k \Vert_{L^2(\gamma)}^2 +  \Vert k \Vert^\Theta_{L^2(\gamma)} \Big)
\\ \leq & C(\E(\gamma_0))
\end{align}
for some exponent $\Theta>2$ and constant $C$ which depends on $\ell(\gamma)$.

Hence, by integrating, it follows
\be 
 \int_\gamma \vert \pat^\perp v \vert ^2 \de s \leq C(\E(\gamma_0)) \,.
\ee
\end{proof}
\begin{prop}\label{propfinaleNavierJ=6}
Let $\gamma_t$ be a maximal solution to the elastic flow of curves subjected to Navier boundary conditions with initial datum
$\gamma_0$, which satisfies the {\em uniform non-degeneracy condition}~\eqref{defunifnondeg} in the maximal time interval $[0,T_{\max})$.
Then, for all $t\in (0,T_{\max})$ it holds
\begin{equation}\label{stimaNavierj=6fin}
\int_{\gamma} \vert\pas^6 k\vert^2\de  s
\leq C ( \mathcal{E}(\gamma_0)) \, .
\end{equation}
\end{prop}
\begin{proof} From formula~\eqref{dt2gammaesteso} and Lemma~\ref{Tpol}, it follows
  \be 
  \pat^\perp v = \pat v ^\perp \nu = \Big(\pas^6k + \pol_7^4(k)+ \mu \pol_5^4 (k) + \mu^2 \pol_3^2(k) \Big) \nu \, .
  \ee
  However, since we are assuming that $\mu$ is constant, we simply have
    \be \label{vlot}
 \pas^6k =\pat v ^\perp + \pol_7^4(k)+ \pol_5^4 (k) +  \pol_3^2(k) 
  \ee
  and by means of Peter-Paul inequality, we get
\be\label{pas6kpol}
  \int_\gamma \vert \pas^6 k \vert^2 \de s \leq \int_\gamma \vert  \pat v ^\perp \vert^2 \de s + C \Bigg(\int_\gamma \vert \pol_7^4(k)\vert^2 \de s+\int_\gamma \vert \pol_5^4(k)\vert^2 \de s+\int_\gamma \vert \pol_3^2(k)\vert^2 \de s\Bigg) \, .
  \ee
We now estimate separately the integrals involving the polynomials. 
\\We start considering
\be
\int_\gamma \vert \pol_7^4(k)\vert^2 \de s= \int_\gamma \Big \vert \prod_{l=0}^4 (\pas^l k )^{\alpha_l}\Big \vert^2 \de s
\ee
where $\alpha_l \in \N$ and $\sum_{l=0}^4 (l+1)\alpha_l = 7$. 
So, by H\"older inequality, we get
\be
\int_\gamma \vert \pol_7^4(k)\vert^2 \de s= \int_\gamma \Big \vert \prod_{l=0}^4 (\pas^l k )^{\alpha_l}\Big \vert^2 \de s \leq \prod_{l=0}^4 \Bigg(\int_\gamma \vert  \pas^l k  \vert^{2\alpha_l \beta_l} \de s \Bigg)^{\frac{1}{\beta_l}}=\prod_{l=0}^4  \Vert \pas^l k \Vert_{L^{2\alpha_l \beta_l}(\gamma)}^{2 \alpha_l}
\ee
where
$\beta_l:= \frac{7}{(l+1)\alpha_l}>1$ if $\alpha_l\ne 0$ (if $\alpha_l=0$ we simply have the integral of a unitary function), which clearly satisfies 
$$\sum_{l=0}^4 \frac{1}{\beta_l} = 1\, . $$
Then, we estimate any of such products by the well-known interpolation inequalities (see~\cite{mant5}, for instance),
\be\label{interppaslk}
\Vert \pas^l k \Vert_{L^{2\alpha_l \beta_l}(\gamma)} \leq C \Vert \pas^6 k \Vert_{L^2(\gamma)}^{\sigma_l} \Vert  k \Vert_{L^2(\gamma)}^{1-\sigma_l}+\Vert  k \Vert_{L^2(\gamma)}
\ee
for some constant $C$ depending on $\alpha_l, \beta_l$ and coefficient $\sigma_l$ given by
\be
\sigma_l= \frac{1}{6} \Big ( l - \frac{1}{2 \alpha_l \beta_l}+ \frac{1}{2}\Big) \in \Big[\frac{l}{6}, 1\Big)\, .
\ee
Moreover, we notice that 
\begin{align}
\sum_{l=0}^4 2 \alpha_l \sigma_l = & \sum_{l=0}^4\frac{1}{3} \Big ( \alpha_l (l+1) - \frac{1}{2 \beta_l}- \frac{\alpha_l }{2}\Big)
\\&= \frac{7}{3}-\frac{1}{6}-\frac{\sum_{l=0}^4 \alpha_l}{6} <2 \, ,
\end{align}
where in the last inequality we use the fact that, since $l$, $\alpha_l $ are respectively the order of derivations and the exponents of the derivative in $\pol_7^4(k)$, it follows 
\be
1<\sum_{l=0}^4 \alpha_l \leq 7 \, .
\ee
Then, multiplying together inequalities~\eqref{interppaslk} and applying the Young inequality, we have
\begin{align}
\int_\gamma \vert \pol_7^4(k)\vert^2 \de s \leq & \left ( \Vert \pas^6 k \Vert_{L^2\gamma)}+\Vert  k \Vert_{L^2(\gamma)}\right)^{\sum_{l=0}^4 2 \alpha_l \sigma_l}  \Vert k \Vert_{L^2(\gamma)}^{\sum_{l=0}^4 2 \alpha_l (1-\sigma_l)}
\\\leq & \e \left(\Vert \pas^6 k \Vert_{L^2(\gamma)} +\Vert k \Vert_{L^2(\gamma)}\right)^2+ C(\e) \Vert k \Vert_{L^2(\gamma)}^{\Theta_1}\label{p47}
\end{align}
for some exponent $\Theta_1>2$.

Arguing in the same way, one can check that
\be\label{p45}
\int_\gamma \vert \pol_5^4(k)\vert^2 \de s \leq \e \left (\Vert \pas^6 k \Vert_{L^2(\gamma)} + \Vert k \Vert_{L^2(\gamma)}\right) ^2+ C(\e) \Vert k \Vert_{L^2(\gamma)}^{\Theta_2}
\ee
and 
\be\label{p43}
\int_\gamma \vert \pol_3^2(k)\vert^2 \de s \leq \e\left( \Vert \pas^6 k \Vert_{L^2(\gamma)}+  \Vert k \Vert_{L^2(\gamma)} \right)^2+ C(\e) \Vert k \Vert_{L^2(\gamma)}^{\Theta_3}
\ee
for some exponents $\Theta_2,\Theta_3>2$.
\\Replacing the estimates~\eqref{p47},~\eqref{p45} and~\eqref{p43} in~\eqref{pas6kpol} and moving the small part of $\Vert \pas^6 k \Vert^2_{L^2(\gamma)} $ on the right-hand side, we have
\be
  \int_\gamma \vert \pas^6 k \vert^2 \de s \leq \int_\gamma \vert  \pat v ^\perp \vert^2 \de s + C (\Vert  k \Vert^2_{L^2(\gamma)} +\Vert  k \Vert^\Theta_{L^2(\gamma)})  \, ,
  \ee
  where, as above, $\theta>2$.  
  Then, we conclude using Proposition~\ref{propNavierJ=6} and the energy monotonicity in Proposition~\ref{energydecreases}.
\end{proof}
\section{Long-time existence}
In the following, we adapt the proof of~\cite[Theorem~4.15]{ManPluPozSurvey} to our situation.
\begin{teo}
 Let $\gamma_0$ be a geometrically admissible initial curve. Suppose that $\gamma_t$ is a maximal solution to the elastic flow with initial datum $\gamma_0$ in the maximal time interval $[0,T_{\max})$ with $T_{\max} \in (0, \infty) \cup \{\infty\}$. Then, up to reparametrization and translation of $\gamma_t$, it follows
$$T_{\max}= \infty$$
or at least one of the following holds
\begin{itemize}
\item $\liminf \ell(\gamma_t )\to 0$ as $ t \to T_{\max}$;
\item $\liminf \tau_2 \to 0$ as $ t \to T_{\max}$ at boundary points.
\end{itemize}
\end{teo}
\begin{proof}
Suppose by contradiction that the two assertions in the statement are not fulfilled and that $T_{\max}$ is finite. So, in the whole time interval $[0,T_{\max})$
the length of the curves $\gamma_t$ is uniformly bounded from below away from zero and the uniform condition~\eqref{unifconditiontau} is satisfied. 
Moreover, since the energy~\eqref{energy} decreases in time, both the $L^2$-norm of the curvature and the length of $\gamma$ are uniformly bounded from above. Let $\e >0$ be fixed, by means of Proposition~\ref{propNavierj=2} and Proposition~\ref{propNavierJ=6} we have that 
\be
\pas^2 k \in L^\infty ([0,T_{\max}); L^2)
\qquad \text{and} \qquad \pas^6 k \in L^\infty ((\e,T_{\max}); L^2)\, .
\ee
Hence, using Gagliardo-Nirenberg 
inequality for all $t\in[0,T_{\max})$
we get
\be
 \Vert \pas^{j}k\Vert_{L^2(\gamma)}
\leq C_1\Vert \pas^{6}k\Vert_{L^2(\gamma)}^\sigma
    \Vert k\Vert_{L^2(\gamma)}^{1-\sigma}+
    C_2 \Vert k\Vert_{L^2(\gamma)}\leq C(\E(\gamma_0))\,,
\ee
for every integer $j \leq 6$, with constants independent on $t$ and for suitable exponent $\sigma$. Actually, by interpolation, we have 
\be
\pas^j k \in L^\infty ((\e,T_{\max}); L^\infty)
\ee
for every integer $j \leq 5$.
Reparametrizing the curve $\gamma_t$ into $\widetilde{\gamma}_t$
with the property $\vert \partial_x\widetilde{\gamma}(x)\vert =\ell(\widetilde{\gamma})$ for every $x\in [0,1]$
and for all $t\in [0,T_{\max})$ and translating so that it remains in a ball $B_R(0)$ for every time (since its length is uniformly bounded from above), we get
\begin{itemize}
    \item $ 0<c\leq \sup_{t\in [0,T_{\max}), x\in [0,1]} \vert
    \partial_x\widetilde{\gamma}(t,x)\vert \leq C<\infty$,
\item
$0<c\leq \sup_{t\in [0,T_{\max}), x\in [0,1]} \vert\widetilde{\gamma}(t,x)\vert \leq C<\infty \, .$\end{itemize}
Hence, $\tau \in L^\infty ([0,T_{\max}); L^\infty)$ and $\pax ^j \widetilde \gamma \in L^{\infty}((\varepsilon,T_{\max});L^\infty)$ for every integer $j \leq 7$. Then, from the observation above and the fact that $\k = k \nu$, we get $\pas^j \k
\in L^{\infty}((\varepsilon,T_{\max});L^\infty)$ for every integer $j \leq 5$ and 
$\pas^6 \k
\in L^{\infty}((\varepsilon,T_{\max});L^2)$. 
Moreover, thanks to our choice of parametrization, we have  
\be
\k (x)= \frac{\pax^2 \widetilde \gamma(x) }{\ell (\widetilde \gamma)^2} \qquad \text{and} \qquad \pas^j \k (x)=\frac{\pax^{j+2} \widetilde \gamma(x) }{\ell (\widetilde \gamma)^{j+2}} \, .
\ee
So, it follows that $\pax^j \widetilde \gamma \in L^{\infty}((\varepsilon,T_{\max});L^\infty)$
for every integer $1 \leq j \leq 7$ and $\pax^8 \widetilde \gamma \in L^{\infty}((\varepsilon,T_{\max});L^2)$.
\\Then, by Ascoli-Arzel\`a Theorem, there exists a curve $\gamma_{\max}$ such that
\be
\lim_{t\nearrow T_{\max}} \pax^j \widetilde{\gamma}(x)=\pax ^j \gamma_{\max}(x)
\ee
for every integer $j \leq 6$. The curve $\gamma_{\max}$ is an admissible initial curve, since by continuity of $k$ and $\pas^2 k$ it fulfills the system~\eqref{navierbc} and uniform condition~\eqref{unifconditiontau} at boundary points.
Then, there exists an elastic flow $\overline  \gamma_t \in C^{\frac{4+\alpha}{4}, 4+ \alpha}\left([T_{\max}, T_{\max} + \delta) \times [0,1]; \R^2\right)$ with $\delta>0$. We again reparametrize $\overline \gamma_t$ in $\hat \gamma_t$ with constant speed equal to length and we have 
\be
\lim_{t\searrow T_{\max}} \pax^j \hat{\gamma}(x)=\pax ^j \gamma_{\max}(x)
\ee
for every integer $j \leq 6$.

Then, 
$$
\lim_{t\nearrow T_{\max}} \partial_t\widetilde{\gamma}(t,x)
=\lim_{t\searrow T_{\max}}\partial_t\hat{\gamma}(t,x)\,.
$$
Thus, we found a solution to the elastic flow in $C^{\frac{4+\alpha}{4}, 4+ \alpha}\left([0, T_{\max} + \delta) \times [0,1]; \R^2\right)$. This obviously contradicts the maximality of $T_{\max}$.
\end{proof}
We conclude by emphasizing that, even if those arguments and techniques have been already used in literature, all the previous works deal with closed curves (see for instance~\cite{kuschatdz,ManPoz,Poz22})
or open curves with fixed boundary points (see for instance~\cite{NoOk14,NoOk17,DaPoSp16,Ch12,Sp17}). So, all the complications that appear in this paper are due to the fact that we have partial conditions on the boundary points.

\section{Appendix}
For the sake of completeness, we show the smoothness of critical points of functional $\E$. 
\begin{lemma}[{\cite[Corollary~6.13, Exercise~6.7]{AltFunctionalAnalysis2016}}]\label{LemmaAltDuality}
 Suppose that $\Omega\subset \R^n$ is open, $f\in L_{loc}^1(\Omega)$, $p\in (1,\infty]$, $1/p + 1/p'=1$, $m\in \N_0$, and that there exists a constant $C_0$ such that for all $k\in \N_0$ with $k\leq m$ and all $\zeta \in C_c^\infty(\Omega)$
 \begin{align*}
 \left \vert \int_\Omega f \partial^k \zeta \de x \right \vert \leq C_0 \| \zeta \|_{L^{p'}(\Omega)}\,.
 \end{align*}
Then$f\in W^{m,p}(\Omega)$ and there exists a constant $C=C(m,C_0)$ with $\| f \|_{W^{m,p}}\leq C$.
\end{lemma}
\begin{prop}[Regularity for critical point of $\E$]\label{propregularitycriticalpoint}
Suppose that $\gamma$ is a critical point of $\E$, then $\gamma$ is of class $C^\infty$. Moreover, for all $l \in \N$ there exists a constant $C_l= C_l(\| \gamma \|_{H^2} )$ such that
\be\label{gammaWl+2inftyestimate}
 \| \gamma \|_{W^{l+2,\infty}} \leq C_l(\| \gamma \|_{H^2} )\,.
\ee
\end{prop}
\begin{proof}
In order to show the regularity of a critical point of the elastic energy, we follow a bootstrap argument based on Lemma~\ref{LemmaAltDuality} (see~\cite{BraDolPlu21} for a similar proof).

Indeed, we prove that for any $m \in \N_0$, $\eta: [0,1] \to \R$ of class $C^\infty$ and $l\in \N_0$, $l \leq m$, we have
\be \label{assertion1}
\int_\gamma k \pa^l_s \eta \de s \leq C(\|\gamma \|_{H^2}) \| \eta \|_{L^1} \,.
\ee
Then, by Lemma~\ref{LemmaAltDuality} we conclude that $\k \in W^{m,\infty}$ and $\gamma \in W^{m+2, \infty}$, where $\k=k \nu$.

We start showing the assertion for $m=1$. We recall that, since $\gamma$ is a critical point of $\E$, it holds
\be\label{firstvar}
\int_{\gamma} 2 \langle \k, \partial^2_{s} \psi \rangle \de  s   
 + \int_{\gamma} (-3 |\boldsymbol{\kappa}|^2+\mu)
\left\langle \tau,\partial_s\psi\right\rangle \de  s=0
\ee
for all $\psi : [0,1] \to \R^2$ of class $H^2$ such that 
\be\label{condpsi}
\psi(0)_2=0 \qquad \text{and} \qquad \psi(1)_2=0\, .
\ee
Moreover, the fact that $\gamma \in H^2$ ensures that the $L^2$-norm of the curvature is bounded, that is
\be
\| \k \|_{L^1} \leq C \| \k \|_{L^2} \leq C( \| \gamma\|_{H^2}) \, .
\ee
We now denote by $F(\gamma, \psi)$ the second integral in~\eqref{firstvar}, so we have
\be \label{estimateF}
\vert F(\gamma, \psi) \vert \leq C (\| \k \|_{L^2}^2 + \mu \ell(\gamma) ) \| \partial_{s} \psi\|_{L^\infty} \leq C(\|\gamma \|_{H^2} ) \| \psi \| _{W^{2,1}} \, .
\ee 
In order to show the $L^\infty$-regularity of $\k$, we consider $\eta \in C^\infty$ and we use
\be \psi(x) = \ell(\gamma)^2\int_0^x \int_0^y  \eta (t) \nu(t) \de t \de y + \ell(\gamma)^2 x \int_0^1 \int_0^y  \eta (t) \nu(t) \de t \de y \ee
as test function in~\eqref{firstvar}. It clearly follows that $\psi \in H^2$ and $\partial^2_s \psi = \eta \nu$ (using the relation $\vert \gamma'(x) \vert = \ell(\gamma)$ for all $x \in [0,1])$.
Then, if we replace $\psi$ in~\eqref{firstvar} we have
\be \int_{\gamma} 2 \langle \k, \eta \nu \rangle \de s   = - F(\gamma , \psi) \ee 
for all $\eta \in C^\infty$. Hence, using the estimate~\eqref{estimateF}, we obtain
\be   \int_{\gamma} 2 \langle \k, \eta \nu \rangle  \de s=  \int_{\gamma} 2 k \eta \de s \leq C(\|\gamma \|_{H^2} ) \| \psi \| _{W^{2,1}} \leq  C(\|\gamma \|_{H^2} ) \| \eta \| _{L^1} \, , \ee 
and by Lemma~\ref{LemmaAltDuality}, we conclude that $\k \in L^\infty$ (that is $\gamma \in W^{2, \infty}$) and there exists a constant $C_0=C_0(\|\gamma \|_{H^2} )$ such that 
\be\label{Linfbound}
\| \k \| _{L^\infty}\leq C_0(\|\gamma \|_{H^2} )\, .\ee
Arguing in the same way, we want to show that $k \in W^{1,\infty}$.
For $\eta \in C^\infty$, we use
\be \psi(x)= \ell(\gamma) \int_0^x \eta (t) \nu (t) \de t+  \ell(\gamma) x \int_0^1 \eta (t) \nu (t) \de t\ee
as test function in~\eqref{firstvar}. So we have $\psi \in H^2$, $\pa_s \psi = \eta \nu$ and
\be \scal{\pa^2_s \psi,\nu}= \scal{\pa_s \eta \nu,\nu}= \pa_s \eta \, . \ee 
Then, relation~\eqref{firstvar} can be written as
\be \int_\gamma k \pa_s \eta \de s = \int_{\gamma} (3 |\k|^2-\mu)
\scal{ \tau,\partial_s\psi}\de s \leq C(\|\gamma \|_{H^2} ) \|\partial_s\psi \|_{L^1} \leq C(\|\gamma \|_{H^2} ) \|\eta \|_{L^1}
\ee
where we used the $L^\infty$-bound in~\eqref{Linfbound}. Then, by Lemma~\ref{LemmaAltDuality} it follows that $\k \in W^{1,\infty}$ (that is $\gamma \in W^{3,\infty})$ and there exists a constant $C_1=C_1(\|\gamma\|_{H^2})$ such that
\be\label{W1infbound}
\| \k \| _{W^{1,\infty}}\leq C_1(\|\gamma \|_{H^2} )\, .\ee
Once we show the assertion for $m=1$, we can suppose that $m\geq 2$ and that it holds for $m-1$. So, we only need to prove the estimate~\eqref{assertion1} for $l=m$.

For $\eta\in C^\infty$, we use $\psi = \partial_s^{l-2} \eta \nu$ as a test function in~\eqref{firstvar}. Hence, we have \begin{align*}
 \scal{\partial^2_{s}\varphi, \nu} & =  \scal{ \partial_{s}(\partial_s^{l-1} \eta \nu + \partial_s^{l-2}\eta \partial_s \nu), \nu} = \partial_{s}^l \eta + 2  \scal{ \partial_{s}^{l-1}\eta \partial_s \nu, \nu } +  \scal{ \partial_s^{l-2}\eta \partial^2_{s}\nu, \nu } \\ & = \partial_{s}^l \eta -  \scal{\partial_s^{l-2} \eta \partial_{s}(k \tau), \nu} = \partial_{s}^l \eta - k^2 \partial_s^{l-2}\eta \,.
\end{align*}
and
\be 
\scal{\partial_s \psi, \tau } = \scal{\partial_s^{l-1} \eta \nu +\partial_s^{l-2} \eta \partial_s \nu, \tau} = -k \partial_s^{l-2} \eta \, .
\ee
Replacing this relations in~\eqref{firstvar}, we obtain
\be
\int_\gamma k\partial_{s}^l \eta \de s = \int_\gamma  k^3 \partial_s^{l-2}\eta  \ds -  \int_{\gamma} k (3 k^2- \mu)\partial_s^{l-2} \eta  \de s \, .
\ee 
In view of the regularity already established, we may integrate by parts the terms involving derivatives of $\eta$ on the right-hand side and we obtain 
\be 
\int_\gamma k\partial_{s}^l \eta \de s \leq C(\|\gamma \|_{H^2}) \, .
\ee 
Since this estimate holds for all $l\leq m$, by Lemma~\ref{LemmaAltDuality}  we conclude that $\k \in W^{m, \infty}$ and there exists a constant $C_l=C_l(\|\gamma \|_{H^2})$ such that estimate~\eqref{gammaWl+2inftyestimate} holds.
\end{proof}

\section*{Acknowledgements}
The author is grateful to Alessandra Pluda and Carlo Mantegazza for several discussions and helpful suggestions.
\medskip \\The author is partially supported by the INdAM–GNAMPA 2022 Project {\em Isoperimetric problems:
variational and geometric
aspects}. 

\bibliographystyle{plain}
\bibliography{ElasticCurve}
\end{document}